\documentclass[11pt,oneside,english,reqno]{amsart}
\usepackage[T1]{fontenc}
\usepackage[latin9]{inputenc}
\usepackage{color}
\usepackage{babel}
\usepackage{amstext}
\usepackage{float}
\usepackage{amsthm}
\usepackage{amssymb}
\usepackage{graphicx}

\usepackage[pagewise]{lineno}

\usepackage[unicode=true,pdfusetitle,
 bookmarks=true,bookmarksnumbered=false,bookmarksopen=false,
 breaklinks=false,pdfborder={0 0 0},backref=false,colorlinks=true]{hyperref}
\hypersetup{
citecolor=blue}

\usepackage{lipsum}
\usepackage[bottom]{footmisc}

\usepackage{geometry}
\makeatletter
\numberwithin{equation}{section}
\numberwithin{figure}{section}
\theoremstyle{definition}
\newtheorem{thm}{\protect\theoremname}[section]
  \theoremstyle{definition}
  
  \theoremstyle{definition}
  \newtheorem{defn}[thm]{\protect\definitionname}
  \theoremstyle{definition}
  \newtheorem{lem}[thm]{\protect\lemmaname}
  \theoremstyle{remark}
  \newtheorem{rem}[thm]{\protect\remarkname}
  \theoremstyle{definition}
  \newtheorem{cor}[thm]{\protect\corollaryname}
  \theoremstyle{definition}
  \newtheorem*{thm*}{\protect\theoremname}
  \theoremstyle{definition}
  \newtheorem{prop}[thm]{\protect\propositionname}

\usepackage[latin9]{inputenc}
\usepackage{geometry}
\usepackage{amsmath}
\usepackage{amsthm}

\numberwithin{equation}{section}
\numberwithin{figure}{section}
\usepackage{enumitem}		
 \let\footnote=\endnote
\@ifundefined{lettrine}{\usepackage{lettrine}}{}

\theoremstyle{definition}
\newtheorem{thmx}{Theorem}
\def\a{\alpha}
\def\b{\beta}

\def\s{\sigma}

\def\g{\gamma}

\def\L{\Lambda}

\def\R{\mathbb{R}}

\def\H{\mathbb{H}}
\def\N{\mathbb{N}}
\def\Z{\mathbb{Z}}

\def\ep{\varepsilon}
\def\vphi{\varphi}

\def\a{\alpha}
\def\b{\beta}
\def\s{\sigma}

\def\A{\mathcal{A}}
\def\B{\mathcal{B}}
\def\CC{\mathcal{C}}
\def\H{\mathsf{H}}

\def\g{\gamma}

\def\L{\mathcal{L}}
\def\HH{\mathcal{H}}

\def\F{\mathcal{F}}
\def\P{\mathsf{P}}
\def\R{\mathbb{R}}

\def\CC{\mathcal{C}}
\def\I{\mathrm{I}}
\def\J{\mathrm{J}}
\def\K{\mathrm{K}}
\def\W{\mathcal{W}}

\def\E{\mathcal{E}}
\def\M{\mathcal{M}}

\def\U{\mathcal{U}}

\def\RP{\mathbb{R}\mathbb{P}}

\newcommand{\id}{\text{Id}}

\def\E{\mathcal{E}}

\def\ep{\varepsilon}

\def\hol{H\"older }

\newcommand{\Wloc}{\mathcal{W}_{\text{loc}}}

\newcommand{\Sig}{\Sigma_T}
\newcommand{\slr}{\text{SL}_d(\R)}
\newcommand{\glr}{\text{GL}_d(\R)}
\newcommand{\gltwo}{\text{GL}_2(\R)}

\makeatother
  \providecommand{\corollaryname}{Corollary}
  \providecommand{\definitionname}{Definition}
  \providecommand{\factname}{Fact}
  \providecommand{\lemmaname}{Lemma}
  \providecommand{\propositionname}{Proposition}
  \providecommand{\remarkname}{Remark}
  \providecommand{\theoremname}{Theorem}
\providecommand{\theoremname}{Theorem}

\makeatletter
\def\blfootnote{\gdef\@thefnmark{}\@footnotetext}
\makeatother

\address{School of Mathematics, Institute for Advanced Study,
Princeton, NJ, 08540, USA} 
\email{cbutler@princeton.edu}
\address{Department of Mathematics, The University of Chicago, Chicago, IL 60637, USA} 
\email{kihopark@math.uchicago.edu}

\begin{document}

\title{Thermodynamic formalism of $\gltwo$-cocycles with canonical holonomies}
\author{Clark Butler, Kiho Park}
\begin{abstract}
We study the norm potentials of \hol continuous $\gltwo$-cocycles over hyperbolic systems whose canonical holonomies converge and are \hol continuous. Such cocycles include locally constant $\gltwo$-cocycles as well as fiber-bunched $\gltwo$-cocycles.
We show that the norm potentials of irreducible such cocycles have unique equilibrium states. Among the reducible cocycles, we provide a characterization for cocycles whose norm potentials have more than one equilibrium states.
\end{abstract}
\date{\vspace{-5ex}}
\date{}
\maketitle
\section{Introduction}
\blfootnote{\textup{2010} \textit{Mathematics Subject Classification}: 37D35, 37C40, 37H15}
\blfootnote{Key words and Phrases: subadditive thermodynamic formalism, norm potentials, linear cocycles, equilibrium states, Gibbs property}

In this paper, we study matrix cocycles over hyperbolic systems and their thermodynamic formalism. Let $(\Sig,\s)$ be a subshift of finite type. For any continuous cocycle $\A \in C(\Sig,\glr)$, we define the associated \textit{norm potential}
$\Phi_\A:= \{\log \vphi_{\A,n}\}_{n \in \N}$ on $\Sig$, where
$$\vphi_{\A,n}(x) = \|\A^n(x)\| \text{ with } \A^n(x):=\A(\s^{n-1}x) \ldots\A(x).$$
The norm $\|\cdot\|$ is the standard operator norm on $\gltwo$.
From the submultiplicativity of the norm, the norm potential $\Phi_\A$ is subadditive: 
$$
\log\vphi_{\A,n+m} \leq  \log \vphi_{\A,n} + \log \vphi_{\A,m} \circ \s^n
$$
for all $m,n \in \N$.

Classical studies of thermodynamic formalism have been successfully extended to subadditive potentials such as $\Phi_\A$; see \cite{barreira1996non} and \cite{cao2008thermodynamic}. Let $\M(\s)$ be the set of $\s$-invariant probability measures. 
Denoting the corresponding subadditive pressure of $\Phi_\A$ by $\P(\Phi_\A)$, the \emph{subadditive variational principle} \cite{cao2008thermodynamic} states
\begin{equation}\label{eq: subadditive var prin}
\P(\Phi_\A) := \sup \{h_\mu(\s) +\F(\Phi_\A,\mu) \colon \mu \in \M(\s)\}
\end{equation}
where 
$$
\F(\Phi_\A,\mu) =\lambda_+(\A,\mu):= \lim\limits_{n\to\infty} \int\frac{1}{n}\log\|\A^n(x)\|~ d\mu(x)
$$
is the \textit{top Lyapunov exponent} of $\A$ with respect to $\mu$.
Any $\s$-invariant probability measure $\mu \in \M(\s)$ attaining the supremum in \eqref{eq: subadditive var prin} is called an \textit{equilibrium state} of $\Phi_\A$.

We will focus on the norm potentials of a class of $\alpha$-\hol $\gltwo$-cocycles $\A$ over $(\Sig,\s)$ with $\a \in (0,1]$ satisfying two extra conditions:
\begin{enumerate}[label=(\alph*)]
\item\label{eq: a}
Denoting the stable and unstable set of $\Sig$ by $\W^{s}$ and $\W^{u}$, the following limits converge: for any $y \in \W^{s}(x)$ and $z \in \W^{u}(x)$,
\begin{equation}\label{eq: canonical hol}
H^s_{x,y} :=\lim\limits_{n \to \infty}\A^n(y)^{-1}\A^n(x) \text{ and }H^{u}_{x,z}:=\lim\limits_{n \to -\infty}\A^n(z)^{-1}\A^n(x).
\end{equation}
When they exists, such $H^{s/u}$ are called the \textit{canonical holonomies} of $\A$.
\item\label{eq: b} The canonical holonomies are \hol continuous with some exponent $\b \in (0,\alpha]$: there exists $C>0$ such that for any $y \in \Wloc^{s}(x) \cup \Wloc^{u}(x)$, we have 
$$\|H^{s/u}_{x,y} - \id\| \leq C\cdot d(x,y)^\b.$$
\end{enumerate}
We denote by $\HH$ the set of $\alpha$-\hol cocycles that meet such requirements: $$\HH:=\{\A \in  C^\alpha(\Sig,\gltwo)  \colon \A \text{ satisfies } \ref{eq: a} \text{ and } \ref{eq: b}\}.$$
There are many natural classes of cocycles that belong to $\HH$. Such cocycles include locally constant cocycles as well as cocycles that are close to being conformal; the later are called the \emph{fiber-bunched} cocycles; see Section \ref{sec: 2} for details. The following definition of irreducibility has previously appeared in the literature, such as in \cite{bochi2019extremal}.
 
\begin{defn}\label{defn: reducible}
A cocycle $\A \in \HH$ is \textit{reducible} if there exists an $\A$-invariant and $H^{s/u}$-invariant line bundle over $\Sig$.
We say $\A \in \HH$ is \textit{irreducible} if $\A$ is not reducible. 
\end{defn}

From the upper semi-continuity of the entropy map $\mu \mapsto h_\mu(\s)$ and the top Lyapunov exponent $\mu \mapsto \lambda_+(\A,\mu)$, the norm potential $\Phi_\A$ of any continuous cocycle $\A \in C(\Sig,\glr)$ has at least one equilibrium state; see \cite{feng2011equilibrium}.
The main result of this paper establishes that irreducibility implies the uniqueness of such equilibrium states for cocycles in $\HH$:

\begin{thmx}\label{thm: A}
Let $\A \in \HH$. If $\A$ is irreducible, then the norm potential $\Phi_\A$ has a unique equilibrium state.	 
\end{thmx}

Theorem \ref{thm: A} is similar to nowadays a folklore result that norm potentials of locally constant cocycles generated by irreducible sets of matrices have unique equilibrium states; see Remark \ref{rem: loc constant and irred then qm}. \textcolor{black}{For fiber-bunched cocycles, Theorem \ref{thm: A} may be obtained by manipulating a result of Bochi and Garibaldi \cite{bochi2019extremal};  we explain such approach in Subsection \ref{subsec: alternate thm a}.  
Theorem \ref{thm: A} applies for a larger class $\HH$ of cocycles, and we establish it via a different method. In order to do so,} we introduce the notion of weakly typical cocycles, which in some sense, applies to most cocycles in $\HH$. 

\begin{defn}\label{defn: weakly typical}
 We say $\A \in \HH$ is \textit{weakly typical} if 
\begin{enumerate}
\item (pinching) There exists a periodic point $p \in \Sig$ such that $\A^{\text{per}(p)}(p)$ has simple eigenvalues of distinct norms with corresponding eigendirections $v_+, v_- \in \R \mathbb{P}^1$;
\item (twisting) There exist $z_+,z_- \in \W^s(p) \cap \W^u(p) \setminus\{p\}$ such that for each $\tau \in \{+,-\}$, the holonomy loop $\psi_p^{z_\tau}:=H^s_{z_\tau,p}\circ H^u_{p,z_\tau}$ twists $v_\tau$:
$$\psi_p^{z_\tau}(v_\tau) \neq v_\tau.$$
\end{enumerate}
\end{defn}
\begin{rem}
Two points $z_+,z_-\in \W^s(p) \cap \W^u(p) \setminus\{p\}$ from the twisting condition above are homoclinic points of $p$.
More generally, given a periodic point $p \in \Sig$, we say $z \in \Sig$ is a \textit{homoclinic point} of $p$ if $z$ belongs to the set $$\H(p):=\W^s(p) \cap \W^u(p) \setminus \{p\}.$$ Equivalently, the homoclinic points of $p$ are characterized as the points other than $p$ whose orbits synchronously approach the orbit of $p$, both in forward and backward time.
\end{rem}

\begin{rem}
The notion of typical cocycles is first introduced by Bonatti and Viana in \cite{bonatti2004lyapunov} for fiber-bunched $\slr$-cocycles.
Weak typicality introduced as in Definition \ref{defn: weakly typical} is weaker than that of \cite{bonatti2004lyapunov}. We elaborate more on such differences in Remark \ref{rem: difference in typical}.
\end{rem}

Define $$\U_w:=\{ \A \in  \HH \colon \A \text{ is weakly typical}\}.$$
Since the canonical holonomies $H^{s/u}$ vary continuously in $\A \in \HH$, $\U_w$ is open in $\HH$. 
Bonatti and Viana \cite{bonatti2004lyapunov} showed that typical cocycles form a dense subset of the set of fiber-bunched $\slr$-cocycles and that $\U_w^c$ has infinite codimension. In fact, the same proof there readily extends to establish the same properties for $\U_w$ considered as a subset of $\HH$.
Next theorem establishes a trichotomy among irreducible cocycles in $\HH$ with weak typicality being one of the alternatives.

\begin{thmx}\label{thm: B}
Suppose $\A \in \HH$ is irreducible. Then either
\begin{enumerate}
\item $\A$ is weakly typical (i.e., $\A \in \U_w$), or
\item there exist two bi-holonomy invariant line bundles interchanged by $\A$, or
\item\label{eq: conformal thm B} there is a \hol conjugacy of $\A$ into the group of linear conformal transformations of $\R^2$.
\end{enumerate}
\end{thmx}

We also study thermodynamic formalism of reducible cocycles in $\HH$. For any $\A \in \HH$, condition \ref{eq: b} implies that the map $(x,y) \mapsto H^{s/u}_{x,y}$ is $\b$-\hol continuous for some $\b \in (0,\alpha]$. Hence, an $\A$-invariant and $H^{s/u}$-invariant line bundle of a reducible cocycle from Definition \ref{defn: reducible} must also be $\b$-\hol continuous. 
By straightening out this line bundle, the reducible cocycle $\A$ admits a $\b$-\hol conjugacy $\CC \colon \Sig \to \gltwo$ such that $\B(x):=\CC(\s x)\A(x)\CC(x)^{-1}$ is upper triangular for every $x \in \Sig$. Since $\Phi_\A$ and $\Phi_\B$ have the same set of equilibrium states, the study of reducible cocycles in $\HH$ reduces to the study of \hol cocycles taking values in upper triangular matrices.

\begin{thmx}\label{thm: C}
Let $\B \in C^\b(\Sig,\gltwo)$ be an $\b$-\hol cocycle taking values in the group of upper triangular matrices:
\begin{equation}\label{eq: B}
\B(x):=\begin{pmatrix}
a(x) & b(x)\\
0 & c(x)
\end{pmatrix}.
\end{equation}
The norm potential $\Phi_\B$ has a unique equilibrium state, unless 
\begin{enumerate}
\item$\log |a|$ is not cohomologous to $\log |c|$, and
\item$\P(\log |a|) = \P(\log |c|).$
\end{enumerate}
If these two conditions hold, then $\Phi_\B$ has exactly two distinct ergodic equilibrium states. 
\end{thmx}

\begin{rem}
Theorem \ref{thm: C} is a more general result than the first two Theorems in the sense that the assumptions are weaker; while we assume that the cocycle $\B$ is \hol continuous with some exponent $\b>0$, we do not require that $\B$ belong to $\HH$.
In particular, once a reducible cocycle $\A \in \HH$ is conjugated to a $\b$-\hol cocycle $\B$ of the form \eqref{eq: B} via the invariant line bundle, extra conditions \ref{eq: a} and \ref{eq: b} on $\A$ no longer play a role in studying thermodynamic formalism of $\Phi_\B$.
\end{rem}

The result for reducible cocycles in $\HH$ is summarized in the following corollary whose proof appears in Section \ref{sec: thm C}.

\begin{cor}\label{cor: reducible0}
Suppose $\A \in \HH$ is reducible and admits a $\b$-\hol conjugacy $\CC \colon \Sig \to \gltwo$ for some $\beta>0$ such that $\B(x):=\CC(\s x)\A(x)\CC(x)^{-1}$ takes values in upper triangular matrices. Then $\Phi_\A$ has a unique equilibrium state unless two conditions from Theorem \ref{thm: C} hold for $\B$, in which case there are two ergodic equilibrium states for $\Phi_\A$. 
\end{cor}

The paper is organized as follows. In Section \ref{sec: 2}, we introduce the setting of our results and survey relevant results in thermodynamic formalism.
Then we prove Theorem \ref{thm: C} in Section \ref{sec: thm C} and Theorem \ref{thm: B} in Section \ref{sec: thm B}. Using these results, Theorem \ref{thm: A} is established in Section \ref{sec: thm A}. 
In Section \ref{sec: 6}, we explain how the results can be applied to the derivative cocycles, restricted to the unstable bundles, of certain Anosov diffeomorphisms.
\\

\noindent\textbf{Acknowledgments}.
The authors would like to thank the anonymous referees who have pointed out a mistake in the initial draft. Their comments have helped improve the paper by a lot. The authors also thank Ben Call, Ping Ngai Chung, and Amie Wilkinson for helpful discussions.

\section{Preliminaries}\label{sec: 2}
\subsection{Symbolic dynamics} Let $T$ be a $q \times q$ square \textit{adjacency matrix} with entries from $\{0,1\}$, and $\Sig$ be the set of bi-infinite $T$-admissible sequences in $\{1,\ldots,q\}^\Z$ defined by
$$\Sig = \{(x_i)_{i \in \Z} \in \{1,\ldots,q\}^\Z \colon T_{x_{i},x_{i+1}} = 1 \text{ for all }i \in \Z\}.$$
Throughout the paper, we will always assume that $T$ is \textit{primitive}, meaning that there exists $N\in \N$ such that all entries of $T^N$ are positive. The primitivity of $T$ is equivalent to the topological mixing property of $(\Sig,\s)$. 
Denoting the left shift on $\Sig$ by $\s$, the dynamical system $(\Sig,\s)$ is called the \textit{subshift of finite type} defined by $T$.
Fix $\theta \in (0,1)$, and we equip $\Sig$ with a metric $d$ defined as follows: for $x = (x_i)_{i\in \Z}$ and $y = (y_i)_{i\in \Z} \in \Sig$, 
$$d(x,y) :=\theta^k$$
where $k$ is the largest integer such that $x_i = y_i$ for all $|i|<k$. Equipped with such metric, $(\Sig,\s)$ becomes a hyperbolic homeomorphism on a compact metric space. In particular, the \textit{local stable set} of $x$ is defined as
$$\Wloc^s(x):=\{y\in \Sig \colon x_i = y_i \text{ for all }i \in \N_0\}.$$
The \textit{stable manifold of $x$} is defined as
$$\W^s(x) := \{y \in \Sig \colon \s^ny \in \Wloc^s(\s^nx)\text{ for some }n \in \N_0\},$$
and is characterized by the set of $y\in \Sig$ such that $ d(\s^nx,\s^ny)\to 0$ as $n$ tends to infinity. 
Similarly, we define the \emph{local unstable set} $\Wloc^u(x)$ as the set of $y\in \Sig$ with $x_i = y_i$ for all $i \leq 0$ and the \emph{unstable set} $\W^u(x)$ as the set of $y \in \Sig$ with $\s^n y  \in\Wloc^u(\s^n x)$ for some $n \leq 0$.

For any $x,y\in \Sig$ with $x_0 = y_0$, we define the \textit{bracket} of $x$ and $y$ by
\begin{equation}\label{eq: local product}
[x,y]:= \Wloc^u(x) \cap \Wloc^s(y) \in \Sig.
\end{equation}

An \textit{admissible word of length $n$} is a word $i_0\ldots i_{n-1}$ with $i_j \in\{1,\ldots ,q\}$ such that $T_{i_j,i_{j+1}}=1$ for each $0 \leq j\leq n-2$. We denote the set of all admissible words of length $n$ by $\L(n)$, and let $\L:=\bigcup\limits_{n=0}^{\infty} \L(n)$ be the set of all admissible words. For any $\I = i_0\ldots i_{n-1} \in \L(n)$, we define the \textit{cylinder defined by $\I$} as
$$[\I]:=\{(x_j)_{n\in \Z} \in \Sig \colon x_j = i_j \text{ for all } 0 \leq j \leq n-1\}.$$

\subsection{Holonomies and cocycles in $\HH$}
In this subsection, we describe natural classes of cocycles that belong to $\HH$. First, we introduce the formal definition of holonomies.

Let $\A \in C(\Sig,\gltwo)$ be a continuous $\gltwo$-valued function. A \textit{cocycle generated by $\A$}, denoted again by $\A$ by an abuse of notation, is the skew product map
\begin{align*}
\A \colon \Sig \times \R^2 &\to \Sig \times \R^2 ,\\
(x,v) &\mapsto (\s x,\A(x)v). 
\end{align*}
For any $n \in \N$, we have $\A^n(x,v) :=(\s^nx,\A^n(x)v)$ where 
$$\A^n(x):= \A(\s^{n-1}x) \ldots \A(x).$$
Note that $\A^n$ satisfies the \textit{cocycle equation}:
$$\A^{m+n}(x)= \A^m(\s^nx)\A^n(x) \text{ for all } m,n \in \N.$$
As $\s$ is invertible, we define $\A^0 \equiv \id$ and $\A^{-n}(x):=\A^n(\s^{-n}x)^{-1}$. Then the above cocycle equation holds for all $m,n \in \Z$.

\begin{defn}\label{defn: local holonomy} A \textit{local stable holonomy} for $\A$ is a family of matrices $H^{s}_{x,y} \in \gltwo $ defined for any $x,y \in \Sig$ with $y \in \Wloc^s(x)$ such that 
\begin{enumerate}
\item $H^s_{x,x} = \id$ and $H^s_{y,z}\circ H^s_{x,y} = H^s_{x,z}$ for any $y,z \in \Wloc^s(x)$,
\item $\A(x) =H^s_{\s y,\s x} \circ  \A(y) \circ H^s_{x,y}$,
\item $H^s \colon (x,y) \mapsto H^{s}_{x,y} $ is continuous as $x,y$ vary continuously while satisfying the relation $y \in \Wloc^s(x)$.
\end{enumerate}
A \textit{local unstable holonomy $H^u_{x,y}$ is likewise defined as above with $s$ replaced by $u$ and $\s$ replaced by $\s^{-1}$ wherever it occurs.}
\end{defn}

Even though a local stable holonomy $H^s_{x,y}$ is defined only for $y \in \Wloc^s(x)$ in Definition \ref{defn: local holonomy}, it can be extended to a \emph{global stable holonomy} $H^{s}_{x,y}$ defined for any $y \in \W^s(x)$ not necessarily in $\Wloc^s(x)$:
$$H^s_{x,y} := \A^n(y)^{-1} H^s_{\s^nx,\s^ny} \A^n(x),$$
where $n \in \N$ is any positive integer such that $\s^ny \in \Wloc^s(\s^nx)$. Similarly, a local unstable holonomy can be extended to a \emph{global unstable holonomy}. Such extension to global holonomies is coherent with the second property from Definition \ref{defn: local holonomy}. It is easily verified that if the canonical holonomies $H^{s/u}$ from \eqref{eq: canonical hol} converge, then they satisfy the properties listed in Definition \ref{defn: local holonomy}.

We now describe a natural class of cocycles, called fiber-bunched cocycles, that belongs to $\HH$. In the following definition, recall that $\theta \in (0,1)$ is the constant defining the metric on $\Sig$.

\begin{defn} \label{defn: fiber-bunched}
Let $\A \in C^\alpha(\Sig,\gltwo)$. We say $\A$ is \textit{fiber-bunched} if for every $x \in \Sig$,
$$\|\A(x)\|\cdot \|\A(x)^{-1}\| \cdot \theta^{\alpha} <1.$$
\end{defn}

Notice that conformal cocycles and their perturbations are fiber-bunched; indeed, fiber-bunched cocycles may be thought of as cocycles close to being conformal.
Let $$C^\alpha_b(\Sig,\gltwo):= \{ \A \in C^\alpha(\Sig,\gltwo) \colon \A \text{ is fiber-bunched}\}$$
be the set of fiber-bunched cocycles.
It is clear from the definition that the set of fiber-bunched cocycles $C^\alpha_b(\Sig,\gltwo)$ is open in $C^\alpha(\Sig,\gltwo)$. 

One important consequence of the \hol continuity and the the fiber-bunching assumption on $\A\in C^\alpha_b(\Sig,\gltwo)$ is the convergence of the canonical holonomies $H^{s/u}$ from \eqref{eq: canonical hol}.
Moreover, for $\A \in C^\alpha_b(\Sig,\gltwo)$ the canonical holonomies vary $\alpha$-\hol continuously (i.e., condition \ref{eq: b} holds with the same exponent as the cocycle $\A$) on the base points \cite{kalinin2013cocycles}: there exists $C>0$ such that for any $y \in \Wloc^{s}(x) \cup \Wloc^{u}(x)$,
$$
\|H^{s/u}_{x,y}-\id\| \leq C \cdot d(x,y)^\alpha.
$$
This shows that the set of fiber-bunched cocycles $C^\alpha_b(\Sig,\gltwo)$ is a subset of $\HH$.

We note that \hol continuity and the fiber-bunching assumption on the cocycle are sufficient but not necessary for the convergence of the canonical holonomies $H^{s/u}$ from \eqref{eq: canonical hol}. For instance, the canonical holonomies $H^{s/u}$ always converge for locally constant cocycles, another natural class of cocycles that belongs to $\HH$.

\begin{defn}
A cocycle $\A$ is \textit{locally constant} if there exists $k \in \N_0$ such that for every $x\in \Sig$, the value of $\A(x)$ depends only on $x_{-k}\ldots x_k \in \L(2k+1)$.
\end{defn}
\begin{rem}
For any locally constant cocycle $\A$ over $\Sig$, by re-coding the base dynamical system $(\Sig,\s)$ into a new subshift of finite type $(\Sigma_{\widetilde{T}},\s)$, we may and we will assume that $\A(x)$ only depends on the zero-th entry $x_0$. Such cocycles are also known as the \textit{one-step cocycles}.	
\end{rem}
For locally constant cocycles, the canonical holonomies $H^{s/u}$ from \eqref{eq: canonical hol} trivially converge to the identity matrix. Hence, locally constant cocycles belong to $\HH$.

Another natural class of cocycles that belongs to $\HH$ is the derivative cocycles of certain Anosov diffeomorphisms restricted to 2-dimensional invariant subbundles. Via Markov partitions, such derivative cocycles can be realized as cocycles over subshifts of finite type, and the results stated in the introduction applies. In Section \ref{sec: 6}, we will discuss such class of cocycles in further details.

We conclude the discussion on holonomies and cocycles in $\HH$ by describing a property called \textit{bounded distortion} that is satisfied by norm potentials of cocycles in $\HH$:
for any $\A \in \HH$, there exists $C\geq 1$ such that for any $n\in \N$, $\I \in \L(n)$, and $x,y \in [\I]$,  
\begin{equation}\label{eq: bdd distortion}
C^{-1}\leq \frac{\vphi_{\A,n}(x)}{\vphi_{\A,n}(y)} = \frac{\|\A^n(x)\|}{\|\A^n(y)\|} \leq C.
\end{equation}

Indeed, the ratio $\|\A^n(x)\|/\|\A^n(y)\|$ is equal to the product of two fractions $\|\A^n(x)\|/\|\A^n(z)\|$ and $\|\A^n(z)\|/\|\A^n(y)\|$ where $z:=[y,x]$ is the bracket \eqref{eq: local product} of $y$ and $x$. For the first fraction, notice that $\A^n(x)$ is equal to $H^s_{\s^nz,\s^nx}\A^n(z)H^s_{x,z}$. From  \hol continuity of the canonical holonomies, the norm of $H^s_{\s^nz,\s^nx}$ and $H^s_{x,z}$ are uniformly bounded above and below independent of $x,z$, and $n$. It then follows that $\|\A^n(x)\|/\|\A^n(z)\|$ is also bounded above and below by a uniform constant. Applying the same argument to $\|\A^n(z)\|/\|\A^n(y)\|$ using instead the unstable holonomy establishes \eqref{eq: bdd distortion}.

Note that locally constant cocycles satisfy the bounded distortion property \eqref{eq: bdd distortion} with the constant $C=1$.

\subsection{Thermodynamic formalism} In this subsection, we briefly survey the theory of both additive and subadditive thermodynamic formalism.

Let $f$ be a homeomorphism on a compact metric space $(X,d)$. A subset $E\subset X$ is \textit{$(n,\ep)$-separated} if any two distinct $x,y \in E$ are at least $\ep$-apart in the $d_n$ metric:
$$d_n(x,y):=\max\limits_{0 \leq i \leq n-1} d(f^ix,f^iy) \geq \ep.$$
From the compactness of $X$, the cardinality of any $(n,\ep)$-separated set $E$ is finite.

For any continuous function $\vphi \colon X \to \R$ (often called a \textit{potential}), the \textit{pressure} $\P(\vphi)$ is defined as
$$
\P(\vphi) := \lim\limits_{\ep\to 0}\limsup\limits_{n \to \infty}\frac{1}{n} \log  \sup \Big\{\sum\limits_{x \in E} e^{S_n\vphi(x)} \colon E \text{ is a }(n,\ep)\text{-separated subset of }X\Big\}
$$
where $S_n\vphi := \vphi+\vphi\circ f  +\ldots+\vphi\circ f^{n-1}$ is the $n$-th Birkhoff sum of $\vphi$. Denoting the set of $f$-invariant probability measures by $\M(f)$, the pressure $\P(\vphi)$ satisfies the following \emph{variational principle}:
\begin{equation}\label{eq: var prin}
\P(\vphi) = \sup\limits_{\mu \in \M(f)} \Big\{h_\mu(f) +\int \vphi ~d\mu\Big\}, 
\end{equation} 
where $h_\mu(f)$ is the measure-theoretic entropy of $\mu$; see \cite{walters2000introduction}. Any $f$-invariant probability measure achieving the supremum in \eqref{eq: var prin} is called an \textit{equilibrium state} of $\vphi$.

The existence and the number of equilibrium states depend on both the potential $\vphi$ and the system $(X,f)$. For instance, any potential over a system whose entropy map $\mu \mapsto h_{\mu}(f)$ is upper semi-continuous has 	at least one equilibrium state; such systems include hyperbolic systems and asymptotically entropy-expansive systems \cite{bowen1972entropy}, \cite{misiurewicz1976topological}.

The question on the finiteness or the uniqueness of the equilibrium state is more subtle. One result along this line is the following theorem of Bowen \cite{bowen1974some} which establishes the uniqueness of the equilbrium states for any \hol potentials over topologically mixing hyperbolic systems.

\begin{prop}\cite{bowen1974some}\label{prop: bowen}
 Let $(\Sig,\s)$ be a mixing subshift of finite type and $\vphi \colon \Sig \to \R$ a \hol continuous function. Then there exists a unique equilibrium state $\mu_\vphi \in \M(\s)$ for $\vphi$, characterized as the unique $f$-invariant measure satisftying the following Gibbs property: there exists $C \geq 1$ such that for any $n \in\N $, $\I \in \L(n)$, and $ x \in [\I]$, we  have
$$
C^{-1} \leq  \frac{\mu_\vphi([\I])}{e^{-n\P(\vphi)}e^{S_n\vphi(x)}} \leq C.
$$
\end{prop}

Although we have stated Proposition \ref{prop: bowen} relevant to our setting of subshift of finite types, Bowen \cite{bowen1974some} established sufficient conditions that guarantee the existence of a unique equilibrium in more general settings, and such conditions have been generalized in many different directions since then. In this paper, we focus on its generalization to subadditive potentials. 

A sequence of non-negative and continuous function $\{\vphi_n\}_{n  \in \N}$ on $X$ is \textit{submultiplicative} if for any $m,n \in \N$, 
$$0 \leq \vphi_{n+m} \leq \vphi_n \cdot \vphi_m\circ f^n.$$
From the submultiplicativity, $\Phi:=\{\log \vphi_n\}_{n \in \N}$ becomes a \textit{subadditive potential} on $X$. 

Given a potential $\vphi \in C(X)$, the Birkhoff sum $\{S_n\vphi\}_{n\in \N}$ is an additive sequence of functions on $X$. In a similar analogy, given a subadditive potential $\Phi=\{\log \vphi_n\}_{n \in\N}$, we may consider the $n$-th function $\log\vphi_n$ of $\Phi$ as a generalization of the $n$-th Birkoff sum $S_n\vphi$ of some potential $\vphi$. By generalizing the definition of the pressure, Cao, Feng, and Huang \cite{cao2008thermodynamic} define the \textit{subadditive pressure} $\P(\Phi)$ of $\Phi$ by 
$$
\P(\Phi) := \lim\limits_{\ep\to 0}\limsup\limits_{n \to \infty}\frac{1}{n} \log  \sup \Big\{\sum\limits_{x \in E} \vphi_n(x) \colon E \text{ is a }(n,\ep)\text{-separated subset of }X\Big\}.
$$

As noted in the introduction, \cite{cao2008thermodynamic} also established that $\P(\Phi)$ satisfies the subadditive variational principle:
$$
\P(\Phi) = \sup \{h_\mu(f) +\F(\Phi,\mu) \colon \mu \in \M(f),~ \F(\Phi,\mu) \neq -\infty\},
$$
where
$$
\F(\Phi,\mu) := \lim\limits_{n\to\infty} \int\frac{1}{n}\log\vphi_n(x)~ d\mu(x).
$$
As in the additive setting, any $\mu\in \M(f)$ attaining the supremum in the subadditive variational principle is called an \textit{equilibrium state} of $\Phi$.

We remark that Barreira \cite{barreira1996non} introduced an alternative way to define a subadditive pressure using open covers. It is not known whether Barreira's definition of the subadditve pressure coincides with Cao, Feng, and Huang's above definition in the most general setting. However, it is shown in \cite{cao2008thermodynamic} that two definitions coincide when the base system is entropy-expansive which includes our setting of subshifts of finite type $(\Sig,\s)$.

In this paper, we will focus on subadditive potentials that arise as norm potentials of $\gltwo$-cocycles over $\Sig$. 
A continuous cocycle $\A \in C(\Sig,\gltwo$) gives rise to the \emph{norm potential} $$\Phi_\A:=\{\log \vphi_{\A,n}\}_{n\in \N}, \text{ where }\vphi_{\A,n}(x)=\|\A^n(x)\|.$$
Because our cocycle $\A$ takes values in $\gltwo$, the Lyapunov exponent $\lambda_+(\A,\mu) = \F(\Phi_\A,\mu)$ is never equal to $-\infty$ for any invariant measure $\mu \in \M(\s)$; this is reflected in the formulation of the subadditive variational principle \eqref{eq: subadditive var prin} in the introduction. Moreover, the subadditive variation principle still holds when the supremum is taken over all ergodic measures instead.

We note that Proposition \ref{prop: bowen} does not readily extend to subadditive potentials. Even restricted to norm potentials $\Phi_\A$ of locally constant cocycles $\A$, there are examples where $\Phi_\A$ admits multiple equilibrium states; see \cite{feng2010equilibrium}. 
Moreover, while fiber-bunched cocycles are nearly conformal, the properties of their norm potentials differ from those of conformal cocycles. The norm potentials of conformal cocycles are additive, and can be studied via tools from classical thermodynamic formalism such as Proposition \ref{prop: bowen}. On the other hand, due to subadditivity, Proposition \ref{prop: bowen} does not necessarily hold for norm potentials of fiber-bunched cocycles without extra assumptions; see Subsection \ref{subsec: qm}.

However, restricted to norm potentials of cocycles in $\HH$, Proposition \ref{prop: bowen} remains valid for norm potentials of large subset of $\HH$. In particular, such subset includes the set of all weakly typical cocycles $\U_w$. See Section \ref{sec: thm A} for more details on the statements in this paragraph.

\begin{defn}
Two continuous functions $\vphi,\psi \in C(\Sig)$ with $P(\vphi) =P(\psi)$ are \textit{cohomologous} if there exists a continuous function $h$ such that $\vphi - \psi = h\circ \s -h$, and we denote it by $\vphi \sim \psi$.
\end{defn}

It is clear from the variational principle \eqref{eq: var prin} that if $\vphi \sim \psi$, then their set of equilibrium states are the same. Since the base dynamic $(\Sig,\s)$ is uniformly hyperbolic, restricted to the class of \hol potentials, this is an if and only if statement; that is, two \hol potentials $\vphi$ and $\psi$ are cohomologous if and only if their unique equilibrium states coincide \cite[Theorem 1.28]{bowen1975ergodic}.

Similarly, if a cocycle $\A \in C(\Sig,\text{GL}_d(\R))$ is continuously conjugated to another cocycle $\B \in C(\Sig,\text{GL}_d(\R))$, then from the subadditive variational principle \eqref{eq: subadditive var prin} the pressures and the set of equilibrium states for the norm potentials $\Phi_\A$ and $\Phi_\B$ are the same. This follows because $\F(\Phi_\A,\mu) = \F(\Phi_\B,\mu)$ for all $\mu \in \M(\s)$ as the norm $\|\CC(x)\|$ of the continuous conjugacy $\CC \colon \Sig \to \gltwo$ is uniformly bounded from the compactness of $\Sig$.
However, it is not necessarily true that two cocycles $\A$ and $\B$ are conjugated to each other just because their equilibrium states coincide.

\section{Proof of Theorem \ref{thm: C}}\label{sec: thm C}
In this section, we will show that the norm potential $\Phi_\A$ of a reducible cocycle $\A \in \HH$ has a unique equilibrium state unless the conjugated cocycle $\B$ as in \eqref{eq: B} satisfies two conditions from Theorem \ref{thm: C}, in which case there are two ergodic equilibrium states for $\Phi_\A$.

For reducible cocycles, we treat them by modifying the results from \cite{feng2010equilibrium}. For locally constant cocycles, Feng and K\"{a}enm\"{a}ki \cite{feng2010equilibrium} showed that after simultaneously conjugating the cocycle into upper block triangle matrices of the same indices such that the tuples of diagonal blocks are irreducible, the number of ergodic equilibrium states for the norm potentials cannot exceed the number of the diagonal blocks. Since the norm potentials of cocycles in $\HH$ have bounded distortion property \eqref{eq: bdd distortion}, we may modify and apply the result of \cite{feng2010equilibrium}.

Denoting by $\mathcal{E}(\s) \subseteq \mathcal{M}(\s)$ the set of ergodic $\s$-invariant probability measures, the following proposition states that the largest Lyapunov exponent of any $\mu \in \mathcal{E}(\s)$ and the pressure of a $\gltwo$-cocycle taking values in upper triangular matrices are coming from the diagonal entries.

\begin{prop}\label{prop: Lyap exp}
Suppose $\B \in C(\Sig,\gltwo)$ is of the form \eqref{eq: B}:
$$\B(x) = \begin{pmatrix}
a(x) & b(x) \\
0 & c(x)
\end{pmatrix}.
$$
Then for any ergodic probability measure $\mu \in \mathcal{E}(\s)$,
\begin{enumerate}[label=(\roman*)]
\item the Lyapunov exponent $\lambda_+(\B,\mu)$ satisfies
$$\lambda_+(\B,\mu) = \max \left\{ \int \log |a| ~d\mu, \int \log|c|~d\mu \right\}.$$
\item $\P(\Phi_\B) = \max \Big\{ \P(\log|a|),\P(\log|c|)\Big\}$.
\end{enumerate}
\end{prop}

In order to prove Proposition \ref{prop: Lyap exp}, we need a lemma from ergodic theory.
\begin{lem}\label{lem: ergodic theory}
Let $(X,\B,\mu)$ be a probability space, $f \colon X \to X$ be an ergodic measure-preserving transformation, and $\vphi \colon X \to \R$ be a $\mu$-integrable function with $\sup\limits_{x \in X} |\vphi(x)|<\infty$. Denoting $\displaystyle \alpha:= \int \vphi ~d\mu$, for any $\ep>0$ and $\mu$-almost every $x\in X$, there exists $n_1 = n_1(x) \in \N$ such that
$$\left|S_n\vphi(f^mx) - n \alpha\right| \leq (n+m) \ep$$ 
for any $n \geq n_1$ and any $m \in \N$.
\end{lem}
\begin{proof}
Let $X_0\subset X$ be a full measure subset from Birkhoff Ergodic Theorem such that the Birkhoff average $\displaystyle\frac{1}{n}S_n\vphi(x)$ converges to $\alpha$ for any $x \in X_0$. 
For each $x \in X_0$, choose $n_0=n_0(x) \in \N$ such that
$$\left| \frac{1}{n} S_n\vphi(x) -\alpha \right| < \ep/2$$
for each $n \geq n_0$. Denoting $a_n:=S_n\vphi(x) - n\alpha$ for each $n\in \N$, define $n_1 =  n_1(x) \geq n_0$ such that
$$n_1 \geq 2/\ep \cdot \left(\max\limits_{1 \leq i \leq n_0-1} |a_i|\right).$$

Consider any $n \geq n_1$ and $m \in \N$. If $m \geq n_0$, then 
\begin{align*}
|S_n\vphi(f^mx) - n\alpha |&= |\left(S_{n+m}\vphi(x) - (n+m)\alpha \right)- \left(S_m\vphi(x)-m\alpha\right)|\\
&\leq (n+2m)\ep/2\\
&\leq (n+m)\ep.
\end{align*}
If $m \leq n_0-1$, then
\begin{align*}
|S_n\vphi(f^mx) - n\alpha |&= |\left(S_{n+m}\vphi(x) - (n+m)\alpha \right)- \left(S_m\vphi(x)-m\alpha\right)|\\
&\leq (n+m)\ep/2+|a_m|\\
&\leq (n+m)\ep
\end{align*}
where the last inequality follows because $n\cdot\ep/2 \geq n_1\cdot \ep/2 \geq |a_m|$.
\end{proof}

\begin{cor}\label{cor: C(x)}
Under the same assumptions of Lemma \ref{lem: ergodic theory}, let $\displaystyle C_0:=\sup\limits_{x \in X} |\vphi(x)|<\infty$. Then for any $\ep>0$ and for $\mu$-almost every $x \in X$, there exists $C(x)>0$ such that
$$|S_n\vphi(f^mx) -n\alpha|<C(x)+(n+m)\ep$$
for all $n,m \in \N$.
\end{cor}
\begin{proof}
In view of Lemma \ref{lem: ergodic theory}, it suffices to set $C(x) = (C_0+|\alpha|)(n_1(x)-1)$ for each $x \in X_0$.
\end{proof}

\begin{proof}[Proof of Proposition \ref{prop: Lyap exp}]
By considering $a(x)$ and $c(x)$ as multiplicative cocycles over $\Sig$, let
$\tau^n(x) := \prod\limits_{i=0}^{n-1}\tau(f^ix)$
for $\tau = \{a,c\}$. Then for any $n \in \N$, we have
$$\B^n(x) = \begin{pmatrix}
a^n(x) & \sum\limits_{i=0}^{n-1} a^{n-i-1}(\s^{i+1}x)b(\s^i x)c^i(x) \\
0 & c^n(x)
\end{pmatrix}.
$$
Denoting the $(i,j)$-entry of a matrix $A$ by $A_{i,j}$, we have
\begin{equation}\label{eq: norm B^nx}
\max\Big\{|\B^n(x)_{1,1}|,|\B^n(x)_{2,2}|\Big\}\leq  \|\B^n(x)\| \leq 2^2\max\limits_{1 \leq i,j\leq 2}|\B^n(x)_{i,j}|.
\end{equation}
Here $\B^n(x)_{1,1} = a^n(x)$ and $\B^n(x)_{2,2} = c^n(x)$.

For any $\ep>0$, Corollary \ref{cor: C(x)} applied to each $\vphi(x) =\log |a(x)|$ and $\vphi(x)  = \log |c(x)|$ gives $C(x)>0$ for $\mu$-almost every $x \in \Sig$ such that
$$|a^{n-i-1}(\s^{i+1}x)| \leq \exp\Big(C(x)+(n-i-1)\int \log |a|~d\mu+n\ep\Big)$$
and
$$|c^{i}(x)| \leq \exp\Big( C(x)+i \int \log |c|~d\mu+i\ep \Big).$$
Denoting $L:= \max\limits_{x \in \Sig} |b(x)|$, we have
\begin{align*}
|\B^n(x)_{1,2}| &= \Big|\sum\limits_{i=0}^{n-1} a^{n-i-1}(f^{i+1}x)b(f^kx)c^i(x) \Big| ,\\
&\leq \sum_{i=0}^{n-1} L\exp\Big(2C(x)+(n-i-1)\int \log |a|~d\mu+i\int \log |c|~d\mu+(n+i)\ep\Big),\\
&\leq n L\exp\Big( 2C(x)+n\max\Big\{ \int \log |a|~d\mu,\int \log |c|~d\mu\Big\} +2n\ep\Big).
\end{align*}
Since $\ep>0$ was arbitrary, it follows from \eqref{eq: norm B^nx} that for $\mu$-a.e. $x\in \Sig$, we have
$$\lambda_+(\B,\mu) = \lim\limits_{n\to\infty} \frac{1}{n} \log \|\B^n(x)\| = \max\Big\{ \int \log |a|~d\mu,\int \log |c|~d\mu\Big\},$$
establishing the first statement of the proposition.

From the first statement and the subadditive variational principle \eqref{eq: subadditive var prin}, the second statement also follows. 
Indeed, let $\{\mu_n\}_{n \in \N}$ be a sequence of measures in $\E(\s)$ such that $h_{\mu_n}(\s)+\lambda_+(\B,\mu_n)$ limits to $\P(\Phi_\B)$. By comparing $\displaystyle \int \log|a|~d\mu_n$ to $\displaystyle \int \log|c|~d\mu_n$ for each $n \in \N$, without loss of generality, we may assume that there exists $n_k \to \infty$ such that $\displaystyle \int \log|a|~d\mu_{n_k} \geq \displaystyle\int \log|c|~d\mu_{n_k}$ for each $k \in \N$. Then from the first statement and the variational priniciple \eqref{eq: var prin}, we have
$$h_{\mu_{n_k}}(\s) + \lambda_+(\B,\mu_{n_k})=h_{\mu_{n_k}}(\s)+\int \log|a|~d\mu_{n_k}\leq \P(\log|a|).$$
From the choice of $\mu_n$, the left hand side limits to $\P(\Phi_\B)$ as $k \to \infty$ and this proves $\P(\Phi_\B) \leq  \max \Big\{ \P(\log|a|),\P(\log|c|)\Big\}$. Conversely, applying similar arguments to $\log|a|$ (i.e., by choosing a sequence $\mu_n \in \E(\s)$ such that $\displaystyle h_{\mu_n}(\s) + \int \log|a|~d\mu_{n}$ limits to $\P(\log|a|)$ and making use of the first statement and the subadditive variational principle \eqref{eq: subadditive var prin}) and $\log|c|$ establishes the reverse inequality.
\end{proof}

If we further suppose that $\B$ from Proposition \ref{prop: Lyap exp} is \hol continuous, then each $\log|a|$ and $\log|c|$ is a \hol potential over a mixing hyperbolic system $(\Sig,\s)$ and has a unique equilibrium state from Proposition \ref{prop: bowen}. Moreover, $\mu_{\log|a|}$ is equal to $\mu_{\log|c|}$ if and only if $\log|a|$ and $\log|c|$ are cohomologous. 
Hence, the following corollary is a consequence of Proposition \ref{prop: Lyap exp}. Also, it is clear that this corollary implies Theorem \ref{thm: C}.

\begin{cor}\label{cor: alternatives}
Suppose $\B \in C^\b(\Sig,\gltwo)$ is of the form \eqref{eq: B}. Then the following holds:
\begin{enumerate}
\item If $\P(\log|a|) \neq \P(\log|c|)$, then $\log|a| \not\sim \log|c|$ and $\Phi_\A$ has a unique equilibrium state. 
\item If $\log|a| \sim \log|c|$, then $\P(\log|a|) = \P(\log|c|)$ and $\Phi_\A$ has a unique equilibrium state $\mu_{\log|a|}=\mu_{\log|c|}$. 

\item If $\log|a| \not\sim \log|c|$ and $\P(\log|a|) = \P(\log|c|)$, then $\Phi_\A$ has exactly two distinct ergodic equilibrium states $\mu_{\log|a|}$ and $\mu_{\log|c|}$.
\end{enumerate}
\end{cor}

\begin{rem}
The third alternative from Corollary \ref{cor: alternatives} is not a vacuous option in that there are cocycles $\B$ satisfying such conditions. For instance, take any two positive \hol continuous functions $\log|a|,~\log|c| \in C^\beta(\Sig,\R^+)$ such that there exist two periodic points $p,q \in \Sig$ of some periods $n,m\in \N$ such that the Birkhoff sum $(S_n\log|a|)(p)$ equals $(S_n\log|c|)(p)$ while $(S_m\log|a|)(q)$ differs from $ (S_m\log|c|)(q)$. The assumption on the Birkhoff sums along the orbit of $q$ ensures that $\log|a|$ is not cohomologous to $\log|c|$. 

If $\P(\log|a|) = \P(\log|c|)$, then by setting $b \equiv 0$, the cocycle $\B$ satisfies the conditions from the third alternative of Corollary \ref{cor: alternatives}. If not, then suppose $\P(\log|a|)> \P(\log|c|)$ without loss of generality. Since $\log|c|$ is a positive function, from the variational principle \eqref{eq: var prin}, $\P(s\log|c|)$ limits to $\infty$ as $s \to \infty$. So there exists $s_0 > 1$ such that $\P(\log|a|) = \P(s_0\log|c|)$, and the assumption on the Birkoff sums along the orbit of $p$ ensures that $\log|a|$ is not cohomologous to $s_0\log|c|$. Then setting $b\equiv 0$ again and replacing the function $\log|c|$ by $s_0\log|c|$, the cocycle $\B$ satisfies the conditions from the third alternative from Corollary \ref{cor: alternatives}. 

We may also choose such functions so that $\B$ is fiber-bunched as well. Indeed, start with any constant function $\log|c| \equiv k$ with $k \in \R^+$ sufficiently large compared to the entropy $h_\text{top}(\s)$ of $(\Sig,\s)$, and let $\log|a|$ be a small perturbation of $\log|c|$ obtained by slightly increasing the function in a neighborhood of some periodic orbit. If the perturbation is small enough, then $s_0$ is sufficiently close to $1$, and the resulting cocycle $\B$ will be fiber-bunched.
\end{rem}

We conclude this section with the proof of Corollary \ref{cor: reducible0}.

\begin{proof}[Proof of Corollary \ref{cor: reducible0}]
In view of Theorem \ref{thm: C}, we only need to show that the statement of Corollary \ref{cor: reducible0} is well-posed, independent of the choice of conjugacy $\CC$. Indeed if there are two \hol conjugacies $\CC_1,\CC_2 \colon \Sig \to \gltwo$ such that both cocycles $\B_i(x):=\CC_i(\s x)\A(x)\CC_i(x)^{-1}$, $i \in \{1,2\}$, take values in the group of upper triangular matrices as in \eqref{eq: B}, then direct computation shows that $\log|a_1| \sim \log|c_2|$ and $\log|a_2| \sim \log|c_1|$. Hence, two conditions from Theorem \ref{thm: C} are intrinsic conditions on the reducible cocycles, independent of the choice of the conjugacy $\CC$. 
\end{proof}	

\section{Proof of Theorem  \ref{thm: B}}\label{sec: thm B}

We now begin the proof of Theorem \ref{thm: B}. The content of Theorem \ref{thm: B} is similar to the fact that a subset of $\gltwo$ which does not preserve a common line either generates a Zariski dense subgroup, preserves a union of two lines, or belongs to a subgroup of the form $O(2) \times \R^*$ in some inner product.

Recall that $\H(p)$ is the set of all homoclinic points of $p$, and for each $z \in \H(p)$, there is an associated holonomy loop $\psi_p^z := H^s_{z,p} \circ H^u_{p,z}$.
As $(\Sig,\s)$ is a mixing hyperbolic system, $\H(p)$ is dense in $\Sig$ for any periodic point $p \in \Sig$.

\begin{lem}\label{lem: irred 1}
Let $\A \in \HH$ be an irreducible cocycle. For any fixed point $p \in \Sig$ and any line $L \in \RP^1$, either
\begin{enumerate}
\item $\A(p)(L) \neq L$, or
\item there exists a homoclinic point $z \in \H(p)$ such that $\psi_p^z (L) \neq L$.
\end{enumerate}
\end{lem}
\begin{proof}
Suppose the conclusion of the lemma does not hold. Then there exists an $\A(p)$-invariant line $L \in \R^2$ that is preserved under $\psi_p^z$ for all homoclinic points $z \in \H(p)$. For each homoclinic point $z \in \H(p)$, we define
$$L_z := H^s_{p,z} (L)= H^u_{p,z}(L).$$
The second equality holds because $L$ is invariant under $\psi_p^z$. 

We will show that such extension of $L$ to $\H(p)$ is \hol continuous. Suppose $x,y \in \H(p)$ with $d(x,y)$ small. Setting $z:=[y,x]$, $z$ is also a homoclinic point of $p$. Then $H^{s}_{x,z}$ maps $L_x$ to $L_z$:
$$L_z = H^s_{p,z}(L) = H^s_{x,z} \circ H^s_{p,x}(L) = H^s_{x,z}(L_x).$$
Similarly, $L_z = H^u_{y,z}(L_y)$.
Hence, $$L_y  = H^u_{z,y}\circ H^s_{x,z}(L_x).$$
Since $H^{s/u}_{x,y}$ varies $\b$-\hol continuously in $x$ and $y$ from condition \ref{eq: b}, there exists $C>0$ depending only on $\A$ such that
$$\rho(L_x,L_y) \leq C d(x,y)^\b,$$
where $\rho$ is the angular distance on $\RP^1$.
 
Since $\H(p)$ is dense in $\Sig$, it follows that $L$ can be uniquely extended to an $\A$-invariant and $H^{s/u}$-invariant line bundle over $\Sig$, contradicting the irreducibility assumption on $\A$.
\end{proof}

The following corollary is an immediate consequence of Lemma \ref{lem: irred 1}.
\begin{cor}\label{cor: irred 1} Let $\A \in \HH$ be an irreducible cocycle, $p \in \Sig$ be a fixed point, and $L \in \RP^1$ be an eigendirection of $\A(p)$. Then there exists $z \in \H(p)$ such that $\psi_p^z (L)\neq L$.
\end{cor}

For any $\A \in C(\Sig,\glr)$ and $\mu \in \mathcal{E}(\s)$, Kingman's subadditive ergodic theorem ensures that the top Lyapunov exponent $\lambda_+(\A,\mu)$ of $\A$ with respect to $\mu$ satisfies
$$
\lambda_+(\A,\mu) = \lim\limits_{n \to \infty} \frac{1}{n}\log\|\A^n(x)\| \text{ for }\mu \text{ a.e. }x \in \Sig.
$$
Indeed, this may be taken as the definition of $\lambda_+(\A,\mu)$.
Similarly, the smallest Lyapunov exponent of $\A$ with respect to $\mu$ may be defined as
$$\lambda_-(\A,\mu):= \lim\limits_{n \to \infty} \frac{1}{n} \log \|\A^n(x)^{-1}\|^{-1} \text{ for }\mu \text{ a.e. }x \in \Sig.$$

We say $\lambda_\pm (\A,\mu)$ are the \textit{extremal Lyapunov exponents} of $\A$ with respect to $\mu$.
For any periodic point $p \in \Sig$, we denote by $\lambda_{\pm}(\A,p)$ the extremal Lyapunov exponents of the invariant measure $\mu_p$ supported on the orbit of $p$.

The following proposition from Kalinin and Sadovskaya \cite{kalinin2010linear} produces an $\A$-invariant conformal (not necessarily non-trivial) sub-bundle when the extremal Lyapunov exponents of $\A$ coincide for all periodic points.

\begin{prop}\cite[Proposition 2.1, 2.7]{kalinin2010linear}\label{prop: conformal subbundle} 
Let $f$ be a transitive $C^{2}$ Anosov diffeomorphism on a compact manifold $M$, $\mathcal{E}$ a finite-dimensional vector bundle over $M$, and $\A\colon \mathcal{E} \to \mathcal{E}$ an $\alpha$-\hol linear cocycle. Suppose for every periodic point $p \in M$, the invariant measure $\mu_p \in \mathcal{E}(\s)$ satisfies 
\begin{equation}\label{eq: equal exponents at periodic}
\lambda_+(\A,p) = \lambda_-(\A,p).
\end{equation}
Then either $\A$ preserves an $\alpha$-\hol continuous conformal structure on $\mathcal{E}$ or $\A$ preserves an $\alpha$-\hol continuous proper non-trivial sub-bundle $\mathcal{E}' \subset \mathcal{E}$ and an $\alpha$-\hol continuous conformal structure on $\mathcal{E}'$. 	 
\end{prop}

Although it is not formulated in the statement of Proposition \ref{prop: conformal subbundle}, the assumption \eqref{eq: equal exponents at periodic} has other consequences as well. First of all, it implies that the canonical holonomies $H^{s/u}$ for $\A$ converge and are as regular as the cocycle $\A$ (see the proof of Corollary 3.6 in \cite{kalinin2013cocycles}). Moreover, the sub-bundle $\E'$ from Proposition \ref{prop: conformal subbundle} is $H^{s/u}$-invariant.

For fiber-bunched cocycles, the following proposition from Bochi and Garibaldi \cite{bochi2019extremal} shows that the converse also holds:
\begin{prop}\cite[Corollary 3.5]{bochi2019extremal}\label{prop: bochi-garibaldi}
Let $\A$ be an $\alpha$-\hol fiber-bunched cocycle of a vector bundle $\mathcal{E}$ over a hyperbolic homeomorphism. An $\A$-invariant sub-bundle $\mathcal{F} \subset \mathcal{E}$ is $\alpha$-\hol if and only if it is $H^{s/u}$-invariant.
\end{prop}
\begin{rem}\label{rem: beyond fb}
While \cite[Corollary 3.5]{bochi2019extremal} is stated for fiber-bunched cocycles, the same result holds for $\alpha$-\hol cocycles whose canonical holonomies converge and are $\alpha$-\hol continuous, including $\alpha$-\hol cocycles satisfying \eqref{eq: equal exponents at periodic}. 
Moreover, Proposition \ref{prop: conformal subbundle} and \ref{prop: bochi-garibaldi} readily extend to our setting where the base dynamical system is a mixing subshift of finite type $(\Sig,\s)$.
\end{rem}
Hence, the conclusion of Proposition \ref{prop: conformal subbundle} for $\A\in C^\alpha(\Sig,\gltwo)$ satisfying \eqref{eq: equal exponents at periodic} may be stated as follows: either $\A$ preserves an $\alpha$-\hol continuous conformal structure on $\Sig \times \R^2$ or $\A$ is reducible. The former alternative is equivalent to the existence of an $\alpha$-\hol continuous conjugacy of $\A$ into the group of linear conformal transformations of $\R^2$. 
With this observation at hand, the proof for Theorem \ref{thm: B} now easily follows.

\begin{proof}[Proof of Theorem \ref{thm: B}]
Let $\A \in \HH$ be an irreducible cocycle. We divide the proof into a few cases.
\begin{enumerate}
\item There exists a periodic point $p \in \Sig$ of period $n$ such that $\A^{n}(p)$ has two eigenvalues of distinct absolute values. Let $\B$ be a cocycle over $(\Sig,\s^n)$ defined by $\B(x):=\A^n(x)$. 
\begin{enumerate}
\item In the case where $\B$ is irreducible, then Corollary \ref{cor: irred 1} applies to $p$ which is now a fixed point with respect to $\s^n$. Hence, $\B$ is weakly typical, which then implies that $\A$ is weakly typical. This gives the first alternative of Theorem \ref{thm: B}.
\item 
In the case where $\B$ is reducible, we get the second alternative of Theorem \ref{thm: B}; see Lemma \ref{lem: fixed} below for the proof.
\end{enumerate}
\item The absolute value of two eigenvalues of $\A^{\text{per}(p)}(p)$ are equal for every periodic point $p \in \Sig$. In this case, the assumption \eqref{eq: equal exponents at periodic} is satisfied. 
Proposition \ref{prop: conformal subbundle} and \ref{prop: bochi-garibaldi} then imply that either there exists an $\alpha$-\hol continuous conjugacy of $\A$ into the group of conformal linear transformations of $\R^2$ or $\A$ is reducible. 
Since $\A$ is irreducible, it must be that $\A$ falls into the third alternative of Theorem \ref{thm: B}.
\end{enumerate}
This completes the proof.
\end{proof}

\begin{lem}\label{lem: fixed} Let $\A \colon \Sig \to\gltwo$ be an irreducible fiber-bunched cocycle. Suppose there exists a periodic point $p\in \Sig$ of some period $n\in \N$ such that $\A^n(p)$ has two eigenvalues of distinct absolute values. If $\B:=\A^n$ is reducible, then $\A$ interchanges two bi-holonomy invariant line bundles.
\end{lem}
\begin{proof}
Let $L$ be the bi-holonomy invariant and $\B$-invariant line bundle. Then there are $n$ bi-holonomy invariant (but not $\A$-invariant nor necessarily distinct) line bundles $\{L_1,\ldots, L_n\}$ defined by $L_i:=\A^{i} L$; that is, $L_i(x):=\A^{i}(\s^{-i}x)L(\s^{-i}x)$. Some of these line bundles might coincide with one another, so we denote the distinct lines bundles among them by $\{\L_1,\ldots,\L_k\}$.

Note $k \geq 2$ because otherwise the irreducibility assumption on $\A$ would be violated. By distinct line bundles, we mean that for $i \neq j$, there exists $x \in \Sig$ such that $\L_i(x) \neq \L_j(x)$. In this case, we will show that if $i \neq j$, then in fact $\L_i(x)$ differs from $ \L_j(x)$ at every $x \in \Sig$.\\

\noindent\textbf{Claim:} For $i\neq j$, we have $\L_i(x) \neq \L_j(x)$ for every $x \in \Sig$.

\begin{proof}[Proof of Claim]
Suppose, for the sake of contradiction, that $\L_{i_0}(x) = \L_{j_0}(x)$ for some $i_0\neq j_0 $ and $x=(x_n)_{n\in \Z}\in \Sig$. From bi-holonomy invariance of $\L_i$'s, it follows that $\L_{i_0}$ and $\L_{j_0}$ agree on 1-cylinder $[x_0]$.

Letting $i_1,j_1\in \{1,\ldots,k\}$ be the indices such that $\L_{\chi_1}= \A\L_{\chi_0}$ for $\chi \in \{i,j\}$. From the previous paragraph, it follows that $\L_{i_1}$ and $\L_{j_1}$ agree on all $y$ with $\s^{-1}y\in [x_0]$.
Notice also that $i_1 \neq j_1$ because if they were the same, then this would imply that $\L_{i_0}$ and $\L_{j_0}$ agree everywhere, contradicting the fact that $\L_i$'s are distinct line bundles.

Repeating this argument iteratively for each $m\in \N$ gives distinct indices $i_m,j_m\in \{1,\ldots,k\}$ such that $\L_{\chi_m} = \A\L_{\chi_{m-1}}$ for $\chi \in \{i,j\}$ and that $\L_{i_m}$ and $\L_{j_m}$ agree on all $y$ with $\s^{-m}y \in [x_0]$.

Recall that $(\Sig,\s)$ is a mixing subshift of finite type with $q$ letters defined by a primitive matrix $T$. Letting $m_0\in \N$ be the mixing rate of $(\Sig,\s)$, we can find $y^{(r)}\in [r]$ for each $r \in \{1,\ldots,q\}$ such that $\s^{-m_0} y^{(r)}\in [x_0]$. This implies that $\L_{i_{m_0}}$ and $\L_{j_{m_0}}$ agree at $y^{(r)}$ for each $r \in \{1,\ldots,q\}$. From bi-holonomy invariance of $\L_i's$, two line bundles $\L_{i_{m_0}}$ and $\L_{j_{m_0}}$ agree everywhere on $\Sig$. However, this is a contradiction to the fact that $\L_i$'s are distinct line bundles.
\end{proof}

We now conclude that $k=2$. This is because if $k \geq 3$, then $\A^{n}(p)$ preserves the union of $k$-distinct lines $\{\L_1(p),\ldots, \L_k(p)\}$, and hence, it is conjugated to a conformal linear transformation. However, this is contradictory to the assumption that $\A^{n}(p)$ has two eigenvalues of distinct absolute values.
Therefore, $k=2$ and $\A$ interchanges $\L_1$ and $\L_2$ because otherwise the irreducibility assumption of $\A$ would be violated.
\end{proof}

We end this section by commenting on the differences between the typicality assumptions from \cite{bonatti2004lyapunov} and from Definition \ref{defn: weakly typical}. 

\begin{rem}\label{rem: difference in typical}
In Bonatti and Viana \cite{bonatti2004lyapunov} where typicality was first introduced, a fiber-bunched $\text{SL}_n(\R)$-cocycle is called \textit{typical} if it satisfies the pinching and twisting assumption. While their pinching assumption is analogous to the pinching assumption from Definition \ref{defn: weakly typical}, their twisting assumption is more restrictive.
The twisting assumption of \cite{bonatti2004lyapunov} requires that there exists a single homoclinic point $z \in \H(p)$ whose holonomy loop $\psi_p^z$ twists all eigendirections. In our Definition \ref{defn: weakly typical}, we allow each $v_+,v_- \in \R\mathbb{P}^1$ to be twisted under holonomy loops of different homoclinic points $z_+,z_- \in \H(p)$, respectively. 
Our definition of weak typicality is flexible enough to establish the  trichotomy in Theorem \ref{thm: B}, and yet has enough structures to guarantee the uniqueness of the equilibrium state; see Proposition \ref{prop: typical implies QM} and \ref{prop: qm then unique eq}. 

Moreover, typicality in \cite{bonatti2004lyapunov} implies that the cocycle is necessarily \textit{strongly irreducible}, meaning that there is no finite union of subspaces preserved under the action of the cocycle. However, our definition of weak typicality in Definition \ref{defn: weakly typical} does not necessarily imply that the cocycle is strongly irreducible. Indeed, we only require that $\psi_p^{z_\tau}(v_\tau) \neq v_\tau$ for each $\tau \in \{+,-\}$.	
In particular, it could happen that $\psi_p^{z_+}(v_+)=v_-$ and $\psi_p^{z_-}(v_-)=v_+$, and such possibility is the reason why weakly typical cocycles may fail to be strongly irreducible.

Lastly, our setting is slightly more general than \cite{bonatti2004lyapunov} that we consider cocycles in $\HH$ which contains the set of fiber-bunched cocycles $C^\alpha_b(\Sig,\gltwo)$.
\end{rem}

\section{Proof of Theorem \ref{thm: A}}\label{sec: thm A}
We prove Theorem \ref{thm: A} in this section. We begin by establishing the uniqueness of the equilibrium state for $\Phi_\A$ in the case of (2) and (3) of Theorem \ref{thm: B}. Then we introduce the notion of quasi-multiplicativity, a property satisfied by all weakly typical cocycles, and explain how the uniqueness of the equilibrium state for $\Phi_\A$ follows from it; see Definition \ref{defn: qm}, Proposition \ref{prop: typical implies QM} and \ref{prop: qm then unique eq}.
Then in the subsequent subsection, we sketch the proof of Proposition \ref{prop: typical implies QM}.

\subsection{Unique equilibrium state for $\Phi_\A$ in the case of (2) and (3) of Theorem \ref{thm: B}} Given an irreducible cocycle $\A \in \HH$, it has to belong to one of three cases listed in Theorem \ref{thm: B}. We explain in this subsection how the last two of the three cases imply the uniqueness of the equilibrium states for $\Phi_\A$.

For the second alternative of Theorem \ref{thm: B} where there exist two bi-holonomy invariant line bundles interchanged by the action of $\A$, by conjugating $\A$ if necessary, we may assume that $\A$ takes the following form:
$$\A(x)=\begin{pmatrix}
0 & a(x) \\
b(x) & 0
\end{pmatrix}.$$

Then consider a cocycle $\B$ over $(\Sig,\s^2)$ defined by $\B(x):=\A^2(x)$; then $\B(x)$ is a diagonal matrix with entries given by $a(\s x)b(x)$ and $a(x)b(\s x)$. From Theorem \ref{thm: C}, the norm potential $\Phi_\B$ has a unique equilibrium state unless two additive potentials $\a (x):=\log |a(\s x)b(x)|$ and $\b(x):=\log |a(x)b(\s x)|$ have equal pressures but are not cohomologous with respect to $\s^2$.

If $\Phi_\B$ has a unique equilibrium state, such equilibrium state must also be the unique equilibrium state for $\Phi_\A$. This is because $\P(\Phi_\B,\s^2) =2 \P(\Phi_\A,\s)$, and hence, any equilibrium state for $\Phi_\A$ is an equilibrium state for $\Phi_\B$; see for instance \cite[Lemma 4.10]{call2020k}.

If instead $\Phi_\B$ has two distinct equilibrium states $\mu_1,\mu_2 \in\M(\s^2)$, each corresponding to $\a$ and $\b$, we will show that $\Phi_\A$ has a unique equilibrium state given by the average of $\mu_1$ and $\mu_2$.

\begin{lem} If $\Phi_\B=\Phi_{\A^2}$ has two distinct equilibrium states $\mu_1,\mu_2 \in \M(\s^2)$, then $\Phi_\A$ has a unique equilibrium state given by $\displaystyle \frac{\mu_1+\mu_2}{2}$.
\end{lem}
\begin{proof}
Suppose that $\Phi_\A$ has two distinct equilibrium states $\mu,\nu\in\M(\s)$. Considered as equilibrium states of $\Phi_\B$, each has to be a linear combination of $\mu_1$ and $\mu_2$. From Lemma \ref{lem: s inv} below, this implies that $\mu_1$ and $\mu_2$ are $\s$-invariant. In particular, $\mu_1$ is an equilibrium state for $\a/2$ over $\s$. This follows because $\mu_1$ is an equilibrium state for $\Phi_{\B}$ which is also $\s$-invariant. Then
\begin{align*}
\P(\Phi_{\A^2},\s^2) = \P(\a,\s^2) = h_{\mu_1}(\s^2)&+\int \a\,d\mu_1 = 2\Big(h_{\mu_1}(\s)+\int \a/2\,d\mu_1\Big) \\
&\leq 2\P(\a/2,\s) \leq \P(\a,\s^2)
\end{align*}
where the last inequality is due to the fact that any $\s$-invariant measure, including any equilibrium state for $\a/2$, can be thought of as a $\s^2$-invariant measure. Hence, all inequalities are indeed equalities, and $\mu_1$ is an equilibrium state for $\a/2$ over $\s$. Likewise, analogous argument shows that $\mu_2$ is an equilibrium state for $\b/2$ over $\s$.

However, this is a contradiction to $\mu_1$ and $\mu_2$ being distinct measures because considered as potentials over $(\Sig,\s)$, $\a/2$ and $\b/2$ are cohomologous:
$$\a/2 -\b/2 = h\circ \s -h$$
where $\displaystyle h:=\frac{1}{2}(\log |a| - \log|b|)$. 
Hence $\Phi_\A$ has a unique equilibrium state $\mu_\A \in \M(\s)$.

In order to show that $\mu_\A$ is the average of $\mu_1$ and $\mu_2$, first notice that $\s_*\mu_1$ coincides with $\mu_2$. This follows because $\s_*\mu_1$ is an ergodic equilibrium state for $\Phi_\B$ distinct from $\mu_1$ (if it were equal to $\mu_1$ itself, then $\mu_1$ would be $\s$-invariant, and by applying the same argument to $\mu_2$, we would end up in the contradictory setting of the previous paragraph), so it must be $\mu_2$. Then notice that $\displaystyle \frac{\mu_1+\s_*\mu_1}{2}=\frac{\mu_1+\mu_2}{2}$ is $\s$-invariant from the $\s^2$-invariance of $\mu_1$ and an equilibrium state for $\Phi_\A$. From the uniqueness of the equilibrium state for $\Phi_\A$ established in the previous paragraph, it follows that $\mu_\A$ is equal to $\displaystyle \frac{\mu_1+\mu_2}{2}$.
\end{proof}

\begin{lem}\label{lem: s inv} Let $\mu_1,\mu_2$ be $\s^2$-invariant. If there exist more than one $\g \in [0,1]$ such that $\g \mu_1+(1-\g)\mu_2$ is $\s$-invariant, then $\mu_1,\mu_2$ are $\s$-invariant.
\end{lem}
\begin{proof}
Suppose there exist distinct $\g_1,\g_2 \in [0,1]$ such that both $\mu = \g_1 \mu_1+(1-\g_1)\mu_2$ and $\nu = \g_2 \mu_1+(1-\g_2)\mu_2$ are $\s$-invariant. 

If one of $\g_1$ and $\g_2$ is 0 or 1, suppose without loss of generality that $\g_1=1$, then $\mu= \mu_1$ is $\s$-invariant. Since both $\nu$ and $\mu_1$ is $\s$-invariant, so is $\mu_2$.

If neither $\g_1$ and $\g_2$ are 0 nor 1, then we have
$\displaystyle \frac{1}{\g_i}\mu = \mu_1 +\frac{1-\g_i}{\g_i}\mu_2$ for $i = 1,2$ and by subtracting the equation for $i=1$ from the equation for $i=2$
we get that
$$\frac{1}{\g_2}\nu  - \frac{1}{\g_1}\mu = \Big(\frac{1-\g_2}{\g_2} - \frac{1-\g_1}{\g_1}\Big)\mu_2.$$
The left hand side is $\s$-invariant and the coefficient of $\mu_2$ is nonzero (as $\g_1 \neq \g_2$), we have that $\mu_2$ is $\s$-invariant. Similarly, we get that $\mu_1$ is also $\s$-invariant. 
\end{proof}

For the third alternative of Theorem \ref{thm: B} where there exists a $\b$-\hol conjugacy $\CC \colon \Sig \to \gltwo$ such that $\B(x) = \CC(\s x)\A(x)\CC(x)^{-1}$ is conformal, the conformality of $\B$ implies that the norm of $\B^n$ is \textit{multiplicative}:
$$\|\B^n(x)\| = \prod_{i=0}^{n-1}\|\B(\s^ix)\|$$
for any $x \in \Sig$ and $n \in \N$. Then $\Phi_\B=\{\log \|\B^n(\cdot )\|\}_{ n\geq 0}$ becomes a \hol continuous additive cocycle generated by $\vphi_\B(x):=\log\|\B(x)\|$ in the sense that $S_n\vphi(x) = \log\|\B^n(x)\|
$. Hence, $\Phi_\B$ has a unique equilibrium state $\mu \in \M(\s)$ from Proposition \ref{prop: bowen}. Since $\A$ and $\B$ are conjugated by a continuous conjugacy $\CC$, the set of equilibrium states for their norms potentials are the same, and hence $\mu$ is the unique equilibrium state for $\Phi_\A$.
 
\subsection{Quasi-multiplicativity}\label{subsec: qm}
In view of the previous subsection, what is left to prove in Theorem \ref{thm: A} is how weak typicality guarantees the uniqueness of the equilibrium state for the norm potential. For weakly typical cocycles $\A$, the norm of $\A^n$ is not necessarily multiplicative, but it is close to being multiplicative in the following sense. For every $n \in \N$ and $\I \in \L(n)$, define
$$\|\A(\I)\| := \sup\limits_{x \in [\I]} \| \A^n(x) \|.$$

\begin{defn}\label{defn: qm} We say $\A \in \HH$ is \textit{quasi-multiplicative} if there exist $c>0, k\in \N$ such that for any $\I,\J \in \L$, there exists $\K = \K(\I,\J)$ with $\I\K\J \in \L$ such that $|\K| \leq k$ and
$$\|\A(\I\K\J)\| \geq c\|\A(\I)\| \|\A(\J)\|.$$ 
\end{defn}
\begin{rem}
Quasi-multiplicativity resembles Bowen's specification property \cite{bowen1974some}.
\end{rem}

The following proposition from the second author's previous result \cite{park2020quasi} states that weakly typical coycles are quasi-multiplicative.

\begin{prop}\cite[Theorem A]{park2020quasi}\label{prop: typical implies QM}
Let $\A \in \U_w$ be a weakly typical cocycle. Then $\A$ is quasi-multiplicative.
\end{prop}

\begin{rem}
\cite{park2020quasi} proves Proposition \ref{prop: typical implies QM} in the setting of fiber-bunched $\glr$-cocycles. As noted in Remark \ref{rem: difference in typical}, weak typicality in this paper defined as in Definition \ref{defn: weakly typical} is more general and differs slightly from typicality defined in \cite{bonatti2004lyapunov} and \cite{park2020quasi}. However, Proposition \ref{prop: typical implies QM} still holds for weakly typical cocycles with little modifications, and we briefly sketch the proof in the following subsection. 
\end{rem}

When $\A\in \HH$ is quasi-multiplicative, then the following proposition from Feng \cite{feng2011equilibrium} and Feng and K\"{a}enm\"{a}ki \cite{feng2010equilibrium} establishes the uniqueness of the equilibrium state for the norm potential $\Phi_\A$.

\begin{prop}\cite{feng2011equilibrium,feng2010equilibrium}\label{prop: qm then unique eq}
Suppose $\A \in \HH$ is quasi-multiplicative. Then $\Phi_\A = \{\log \vphi_{\A,n}\}_{n\in\N}$ has a unique equilibrium state $\mu_\A \in \M(\s)$, and $\mu_\A$ has the following subadditive Gibbs property: there exists $C \geq 1$ such that for any $n \in \N$ and $\I \in \L(n)$, we have 
$$
C^{-1} \leq \frac{\mu_\A([\I])}{e^{-n\P(\Phi_\A)}\|\A^n(x)\|} \leq C
$$
for any $x \in [\I]$.
\end{prop}

Note that Theorem \ref{thm: A} now follows from Theorem \ref{thm: B}, Proposition \ref{prop: typical implies QM} and \ref{prop: qm then unique eq}. Before sketching the proof of Proposition \ref{prop: typical implies QM}, a few remarks on Proposition \ref{prop: qm then unique eq} are in order.

\begin{rem}
In \cite{feng2010equilibrium}, Proposition \ref{prop: qm then unique eq} is established for quasi-multiplicative locally constant cocycles over one-sided shifts. The result is then extended for quasi-multiplicative functions $\vphi \colon \L \to \R$ in \cite{feng2011equilibrium}, covering more general class of subadditive potentials. Moreover, the bounded distortion property \eqref{eq: bdd distortion} of the norm potentials for cocycles in $\HH$ allows the result to be extended to norm potentials of quasi-multiplicative cocycles. Such result for fiber-bunching cocycles $C^\alpha_b(\Sig,\glr)$ are established in \cite{park2020quasi}. The only use of fiber-bunching assumption there is to establish the convergence as well as the \hol continuity of the canonical holonomies $H^{s/u}$, and hence, the same result holds for $\A \in \HH$ as well. This is similar in spirit to how Proposition \ref{prop: bochi-garibaldi} is stated and proved for fiber-bunched cocycles in \cite{bochi2019extremal}, but the statement holds for more general cocycles; see Remark \ref{rem: beyond fb}.
\end{rem}

\begin{rem}\label{rem: loc constant and irred then qm}
Given a finite set of matrices $\mathsf{A} := \{A_1,\ldots,A_q\} \subset \glr$, we say $\mathsf{A}$ is \textit{irreducible} if there does not exists a non-zero proper subspace $V\subset \R^d$ such that $A_iV = V$ for all $1\leq i \leq q$. Denoting by $\Sigma^+$ a one-sided full shift generated by $q$ alphabets, Feng and K\"{a}enm\"{a}ki \cite{feng2010equilibrium} established that if $\mathsf{A}= \{A_1,\ldots,A_q\}\subset \glr$ is irreducible and $\A \colon \Sigma^+ \to \glr$ is a locally constant cocycle given by $\A(x) := A_{x_0}$ where $x_0$ is the 0-th entry of $x = (x_i)_{i \in \N_0}$, then the norm potential $\Phi_\A$ is quasi-multiplicative. Hence $\Phi_\A$ has a unique equilibrium state from Proposition \ref{prop: qm then unique eq}. The same result extends via same methods to locally constant cocycles over two-sided full shifts $(\Sigma,\s)$ generated by irreducible sets of matrices.

We also remark that for such $\gltwo$-cocycles $\A$, the irreducibility of its generating set $\mathsf{A}$ as in the paragraph above is equivalent to the irreducibility of $\A$ as in Definition \ref{defn: reducible}. 
Consider a locally constant $\gltwo$-cocycle $\A$ over $\Sigma$ generated by $\mathsf{A}$ which is reducible as in Definition \ref{defn: reducible}. Since $\A$ is locally constant, we have $H^{s/u} \equiv \id$, and from bi-holonomy invariance the corresponding line bundle $L$ must consist of $q$ lines $\{L_i\}_{i=1}^q$ invariant under the action of $\A$. As $\Sigma$ is a full shift, we must have $A_i L_j = L_i$ for all $1 \leq i,j \leq q$. Fixing an $i$ and varying over all $j$ shows that all $L_i$'s are equal and hence $L$ is a constant line bundle. We have just shown that the generating set $\mathsf{A}$ is reducible. The other direction of the equivalence also follows from a similar reasoning.
\end{rem}

\subsection{Proof of Proposition \ref{prop: typical implies QM}}

We now sketch the proof of Proposition \ref{prop: typical implies QM} by following the proof of \cite[Theorem A]{park2020quasi}. The proof outlined below is simpler than the original proof appearing in \cite{park2020quasi} as we are working with $\gltwo$-cocycles. For statements not fully elaborated in the sketch of the proof, we refer the readers to \cite[Section 4]{park2020quasi} for details.

Let $\A \in \U_w$ be a weakly typical cocycle with a periodic point $p \in \Sig$ satisfying the pinching assumption and $z_\pm \in \H(p)$ be the homoclinic points satisfying the twisting assumption. 

We begin by making a few simplifying assumptions. By passing to a suitable power if necessary, we assume that $p$ is a fixed point of $\s$, and set $P:=\A(p)$; see \cite[Lemma 4.9]{park2020quasi}. Note that the weak typicality assumption on $\A$ is still valid after passing to a suitable power. This is because the holonomies $H^{s/u}$ for the cocycle $\A$ over $\s$ coincide with the holonomies for the cocycle $\A^n$ over $\s^n$.

Note that if $z$ belongs to $\H(p)$, then so does any point in its orbit. Moreover, the homoclinic loops of two homoclinic points in the same orbit are conjugated to each other by some power of $P$: for any $z \in \H(p)$ and $r \in \Z$, we have 
$$
\psi_p^z = P^{-r} \circ\psi_p^{\s^r z} \circ P^r.
$$
Hence, if $z_{\pm} \in \H(p)$ satisfies the twisting assumption (i.e., $\psi_p^{z_\tau}(v_{\tau}) \neq v_{\tau}$ for $\tau \in \{+,-\})$, so does $\s^rz_{\pm} \in \H(p)$ for any $r\in \Z$. So we may assume that $z_- $ belongs to $\Wloc^u(p)$ by replacing it by its suitable pre-image under $\s$. Similarly, we may assume that $z_+$ belongs to $\Wloc^s(p)$.

We now set up a few notations. As in the proof of Lemma \ref{lem: irred 1}, let $\rho$ be the angular distance on $\RP^1$. A $\delta$-ball in $\RP^1$ centered at $v$ with respect to $\rho$ will be denoted by $B_\rho(v,\delta)$. For any $A \in \gltwo$, we choose a singular value decomposition
$$A = U\Lambda V^\intercal$$
such that the singular values in $\Lambda$ are listed in a non-increasing order. We define $u(A)$ and $v(A)$ as the first columns of $U$ and $V$, respectively. They are related by the following equation:
$$\|A\|u(A) = Av(A).$$
When there is no confusion, we will not distinguish the notations between vectors in $\R^2$ and their projections to $\RP^1$. 
\begin{rem}\label{rem: linear algebra}
$u(A)$ and $v(A)$ may be thought of as the most expanding direction of $A^*$ and $A$, respectively. 
Using linear algebra, it can be shown that given $\ep >0$, there exists $c>0$ such that for any $A,B \in \gltwo$, 
$$
\pi/2-\rho(v(A),u(B)) >\ep \implies \|AB\| \geq c \|A\|\|B\|.
$$
In fact, a slightly stronger statement of similar flavor is true; see \cite[Lemma 4.5]{park2020quasi}: for any $\ep>0$, we have
$$
\rho(v(A)^\perp,Cu(B))>\ep \implies \|ACB\| \geq \|A\|\|B\| \cdot \sin(\ep)\cdot \|C^{-1}\|^{-1}.
$$
This will be used in establishing quasi-multiplicativity for weakly typical cocycles.
\end{rem}
We say $a,b,c,d \in \Sig$ (in the prescribed order) form a \textit{holonomy rectangle}, and denote it by $[a,b,c,d]$, if 
$$b \in \Wloc^s(a),~c \in \Wloc^u(b),~d \in \Wloc^s(c),\text{ and } a \in \Wloc^u(d).$$
By an \textit{edge} of such a rectangle we mean two adjacent vertices in the rectangle such as $ab$ and $bc$, and by the \textit{length} of an edge we mean the distance between the endpoints of the edge.
For such a rectangle $[a,b,c,d]$, we define $$H[a,b,c,d]:=H^u_{d,a}\circ H^s_{c,d}\circ H^u_{b,c}\circ H^s_{a,b}.$$ 

\begin{lem}\label{lem: m(ep)}
Given $\ep>0$, there exists $m:=m(\ep) \in \N$ such that for any holonomy rectangle $[a,b,c,d]$ and any $v\in \RP^1$,
\begin{enumerate}
\item If one of (hence, a pair of) the edges of the rectangle has length at most $\theta^m$, then
$$\rho(H[a,b,c,d](v),v) <\ep/2.$$
\item If all edges of the rectangle have length at most $\theta^m$, then
$$\rho(H^u_{b,c}\circ H^s_{a,b}(v),v)<\ep/2 \text{ and }\rho(H^s_{d,c}\circ H^u_{a,d}v,v)<\ep/2,$$  
\end{enumerate}
\end{lem}
\begin{proof}
Both claims follow from \hol continuity of the canonical holonomies $H^{s/u}$.	
\end{proof}

For the proof of Proposition \ref{prop: typical implies QM}, we need to make use of the adjoint cocycle $\A_*$ over $(\Sig,\s^{-1})$ defined as follows: for any $x \in \Sig$ and $u,v \in \R^2$,
$$\langle \A_*(x)u, v \rangle =  \langle u, \A(\s^{-1}x)v \rangle.$$
The adjoint cocycle $\A_*$ shares many properties with $\A$. For instance, if $\A$ admits holonomies $H^{s/u}$, then $\A_*$ also admits holonomies $H^{s/u,*}$ given by
$$H^{s,*}_{x,y} = (H^u_{y,x})^* \text{ and }H^{u,*}_{x,y} = (H^s_{y,x})^*.$$ 
Hence, $\A$ is weakly typical if and only if $\A_*$ is weakly typical, with roles of $z_+$ and $z_-$ switched; see \cite[Lemma 4.6]{park2020quasi}.

Let $v_\pm \in \R\mathbb{P}^1$ be the eigendirections of $P$ corresponding to eigenvalues $\lambda_\pm$ with $|\lambda_+| > |\lambda_-|$. Note that $\A_*(p) =P^*$ has the same eigenvalues $\lambda_\pm$ as $P$. Also, the eigendirections of $w_+,w_- \in\R \mathbb{P}^1$ of $P^*$ are equal to $(v_-)^\perp,(v_+)^\perp \in \R\mathbb{P}^1$, respectively. 

Now we fix some constants. Let 
$$\b := \rho(v_+,v_-)>0.$$
Since the top eigendirection $w_+ \in \RP^1$ of $P^*$ is equal to $(v_-)^{\perp}\in \RP^1$, $\b$ is also equal to $\pi/2-\rho(v_+,w_+)>0$.

As $\psi_p^{z_-}(v_-) \neq v_-$ from the twisting assumption, choose $\ep_0 \in (0,\beta/8)$ such that 
\begin{equation}\label{eq: psi twisting}
\rho(\psi_p^{z_-}(v),v_-) >\ep_0 \text{ for any }v \in B_\rho(v_-,\ep_0).
\end{equation}
Note that we will only make use of $z_-$ when working with $\A$; another twisting homoclinic point $z_+$ will be used when working with $\A_*$. 

Let $m:= m(\ep_0)$ from Lemma \ref{lem: m(ep)} and fix $\ell \in \N$ such that
\begin{enumerate}
\item $d(\s^\ell z_-,p) \leq \theta^m$, and
\item $P^\ell v \in B_\rho(v_+,\ep_0/2)$ for any $v \not\in B_\rho(v_-,\ep_0/2)$.
\end{enumerate}
Such choice of $\ell$ is possible because $\s^nz_- \to p$ as $n\to \infty$ and $P$ has simple eigenvalues $\lambda_\pm$ of distinct norms.

Lastly, by decreasing $\ep_0$ and increasing $m,\ell$ in the sequential order they are defined if necessary, we assume that they also work for the adjoint cocycle $\A_*$ with the relevant modifications. By relevant modifications, we mean that \eqref{eq: psi twisting} holds for $\ep_0>0$ with the roles of $z_-$ and $z_+$ switched and the role of $v_-\in \R\mathbb{P}^1$ replaced by $w_- \in\R\mathbb{P}^1$.
Moreover, we mean that $m$ satisfies Lemma \ref{lem: m(ep)} for $H^{s/u,*}$ and the defining properties of $\ell$ hold for $z_+$, $P^*$, and $w_\pm$.
We are done with preliminary set up for the proof of Proposition \ref{prop: typical implies QM}. 

\begin{proof}[Proof sketch of Proposition \ref{prop: typical implies QM}]
Let $\A\in \U_w$ be a weakly typical cocycle, $\ep_0,m,\ell\in\N$ be as above, and $\I,\J \in \L$ be any admissible words. 

Denoting the mixing rate of $\Sig$ by $\bar{\tau} \in \N$ (i.e., $T^{\bar{\tau}} >0)$, we can find $\omega \in [\I]$ such that $\s^{\bar{\tau}+|\I|}(\omega) \in \Wloc^s(p)$. By setting $\tau:=\tau(\I)=\bar{\tau}+m+|\I|$, we ensure that $d(\s^{\tau}\omega,p) \leq \theta^m$. 
\begin{lem} There exists $x\in [\I] \cap \W^s(p)$ such that $\s^\tau x \in\Wloc^u(\s^\tau\omega)$ and that 
$$
u_x:= H^s_{\s^{\tau+\ell}x,p}\circ \A^\ell(x)\circ H^u_{\s^\tau \omega, \s^\tau x}u(\A^\tau(\omega))
$$
belongs to $B_\rho(v_+,\ep_0/2)$. Moreover, $d(\s^{\tau+\ell}x,p)\leq \theta^m$. 
\end{lem}

\begin{proof}
The construction  of $x \in [\I]$ depends on the following direction
$$u_\omega:=H^s_{\s^\tau \omega,p}u(\A^\tau(\omega)) \in \R\mathbb{P}^1.$$

If $u_\omega \in B_\rho(v_-,\ep_0/2)$, let $x:=\s^{-\tau}[\s^\tau \omega,z_-]$, where the bracket operation is defined in \eqref{eq: local product}.
By using the  defining properties of the holonomies and the definition of the holonomy rectangle, $u_x$ is equal to the following expression:
$$u_x = P^\ell \circ \psi_p^{z-}\circ H[p,\s^\tau \omega ,\s^\tau x, z_-] (u_\omega) \in \RP^1.$$
Since $d(\s^\tau w,p) \leq \theta^m$, the first statement of Lemma \ref{lem: m(ep)} applies to the holonomy rectangle $[p,\s^\tau \omega ,\s^\tau x, z_-]$. It then follows from the assumption $u_\omega \in B_\rho(v_-,\ep_0/2)$, and the choice of $\ep_0>0$ and $\ell \in \N$ that $u_x$ belongs to $B_\rho(v_+,\ep_0/2)$. The last claim of the lemma also follows because $d(\s^{\tau+\ell}x,p)=d(\s^\ell z_-,p) \leq \theta^m$ from the choice of $\ell$.

If instead $u_\omega \not\in B_\rho(v_-,\ep_0/2)$, then set $x := \omega$ and consider $u_x \in \R\mathbb{P}^1$. Since $x$ is equal to $\omega$, the unstable holonomy $H^u_{\s^\tau \omega,\s^\tau x}$ is equal to the identity, and $u_x$ simplifies to $P^\ell u_\omega$ from the properties of the holonomies. Then from the assumption $ u_\omega \not\in B_\rho(v_-,\ep_0/2)$ and choice of $\ell \in \N$, it follows that $u_x$ belongs to $B_\rho(v_+,\ep_0/2)$. Moreover, $d(\s^{\tau+\ell}x,p) = d(\s^{\tau+\ell}\omega,p) \leq \theta^{m+\ell}<\theta^m$.

So, in either cases, we have $x \in [\I] \cap \W^s(p)$ with $d(\s^{\tau+\ell}x,p)\leq \theta^m$ such that $u_x \in \RP^1$ is $\ep_0/2$-close to $v_+ \in \RP^1$.
\end{proof}

Similarly, given $\J \in \L$, we apply the same argument to the adjoint cocycle $\A_*$; this produces two points $\iota, y \in \s^{|\J|}[\J]$ where $y \in \s^{|\J|}[\J] \cap \W^u(p)$ is constructed depending on the direction $u_\iota:= H^{s,*}_{\s^{-\tau(\J)}\iota,p}u(\A_*^{\tau(\J)}(\iota)) \in \RP^1$ such that $d(\s^{-\tau(\J)-\ell} y,p)\leq\theta^m$ and that 
$$
u_y:=H^{s,*}_{\s^{-\tau(\J)-\ell}y,p} \circ \A^\ell_*(\s^{-\tau(\J)}y) \circ H^{u,*}_{\s^{-\tau(\J)}\iota,\s^{-\tau(\J)}y} u (\A_*^{\tau(\J)}(\iota)) \in \RP^1
$$
belongs to $B_\rho(w_+,\ep_0/2)$.

Using $x \in [\I]$ and $y \in \s^{|\J|}[\J]$, we create a new point 
$$\chi:=\s^{-\tau(\I)-\ell} [\s^{\tau(\I)+\ell}x,\s^{-\tau(\J)-\ell}y]$$ 
in $[\I]$. 
Since $d(\s^{\tau(\I)+\ell}x,p)$ and $d(\s^{-\tau(\J)-\ell}y,p)$ are at most $\theta^m$, 
the second statement of Lemma \ref{lem: m(ep)} applies to the rectangle $[p,\s^{\tau(\I)+\ell}x,\s^{\tau(\I)+\ell}\chi,\s^{-\tau(\J)-\ell}y]$.
Then the fact that $u_x$ and $u_y$ are $\ep_0/2$-close to $v_+$ and $w_+$, respectively, together the choice of $\ep_0$ show that the $\rho$-distance between the directions 
$$\A^\ell(\s^{\tau(\I)}\chi)\circ H^u_{\s^{\tau(\I)}\omega, \s^{\tau(\I)}\chi} u(\A^{\tau(\I)}(\omega))
$$
and
$$
\A_*^{\ell}(\s^{\tau(\I)+2\ell}\chi) \circ H^{u,*}_{\s^{-\tau(\J)}\iota,\s^{\tau(\I)+2\ell}\chi}v(\A^{\tau(\J)}(\s^{-\tau(\J)}\iota))$$
are bounded away from $\pi/2$ by at least $3\b/4$. By applying the idea described Remark \ref{rem: linear algebra}, we obtain a connecting word $\K \in \L$ of length $k:=2(\bar{\tau}+m+\ell)$ such that 
$$\|\A(\I\K\J)\| \geq c \|\A(\I)\|\|\A(\J)\|$$ 
for some constant $c>0$ independent of $\I,\J \in \L$. See \cite[Section 4]{park2020quasi} for more details.
\end{proof}

\begin{rem}
In fact, \cite[Theorem A]{park2020quasi} proves more than what is stated in Proposition \ref{prop: typical implies QM}. For weakly typical cocycles $\U_w$, the constants $c,k>0$ from quasi-multiplicativity can be chosen uniformly near $\A$. This implies the continuity of the pressure $\P(\Phi_\A)$ and the equilibrium state $\mu_{\A}$ restricted to the set of weakly typical cocycles $\U_w$; see \cite[Theorem B]{park2020quasi}.

In this direction of results, we remark that Cao, Pesin, and Zhao \cite{cao2019dimension} recently established a more general result that the map $\A \mapsto \P(\Phi_\A)$ is continuous on $C^\alpha(\Sig,\glr)$ using a different approach. See also \cite{feng2014non}.
\end{rem}

\subsection{Alternate proof of Theorem \ref{thm: A} for fiber-bunched cocycles}\label{subsec: alternate thm a} 
In this subsection, we explain how Theorem \ref{thm: A} may alternatively be established for fiber-bunched cocycles $C^\alpha_b(\Sig,\gltwo)$ based on the result of Bochi and Garibaldi \cite{bochi2019extremal}. This method still relies on Proposition \ref{prop: qm then unique eq}, but the approach in obtaining quasi-multiplicativity circumvents invoking Theorem \ref{thm: B}. 
 
Bochi and Garibaldi \cite[Proposition 3.11]{bochi2019extremal} showed that irreducible and strongly fiber-bunched automorphisms of \hol vector bundles over hyperbolic homeomorphisms are uniformly spannable. While they established this results for other usages, we explain below how it applies to thermodynamic formalism of the norm potentials of $\glr$-cocycles.

In our context of $\gltwo$-cocycles, strongly fiber-bunching condition coincides with the usual fiber-bunching condition defined as in Definition \ref{defn: fiber-bunched}. For $\glr$-cocycles with $d\geq 3$,  strong fiber-bunching requires that for all $x \in \Sig$,
$$\|\A(x)\|\cdot\|\A(x)^{-1}\|\cdot \theta^\beta<1$$
for some $\beta \in (0,\alpha)$ depending only on the base dynamic system $(\Sig,\s)$ and the \hol exponent $\alpha$ of $\A$; see \cite{bochi2019extremal} for precise definition.

While irreducibility introduced in Definition \ref{defn: reducible} is formulated for $\gltwo$-cocycles, the definition can be easily extended for $\glr$ cocycles for any $d \in\N$:
\begin{defn}\label{defn: irreducible 2} A fiber-bunched $\glr$-cocycle over $\Sig$ is \textit{irreducible} if there does not exist a proper $\A$-invariant and $H^{s/u}$-invariant sub-bundle.
\end{defn}
Note also that quasi-multiplicativity defined in Definition \ref{defn: qm} also makes sense for $\glr$-cocycles for any $d \in \N$. We now introduce the notion of spannability for fiber-bunched $\glr$-cocycles.

\begin{defn}\cite[Section 3.4]{bochi2019extremal} A fiber-bunched $\glr$-cocycle $\A$ over $\Sig$ is \textit{spannable} if for any $x,y\in \Sig$ and $u \in \R^d$, there exists
\begin{enumerate}
\item $x_1,\ldots,x_d \in \W^u(x)$, and
\item $n_1,\ldots,n_d \in \N_0$ such that the points $y_i:=\s^{n_i}x_i$ all belong to $\W^s(y)$,
\end{enumerate}
with the property that $\{v_i\}_{i=1}^d$ defined by 
\begin{equation}\label{eq: v_i}
v_i := H^s_{y_i,y}\circ \A^{n_i}(x_i)\circ  H^u_{x,x_i} (u)
\end{equation}
forms a basis of $\R^d$.
\end{defn}

In the context of $\glr$-cocycles, the relevant result of \cite{bochi2019extremal} can be formulated as follow:
\begin{prop}\cite[Theorem 3.7]{bochi2019extremal}\label{prop: spannable} 
Let $\A \colon \Sig \to \glr$ be a strongly fiber-bunched and irreducible cocycle. Then $\A$ is spannable. 
\end{prop}
In fact, Bochi and Garibaldi proved a stronger statement \cite[Proposition 3.9]{bochi2019extremal} under the same assumptions of Proposition \ref{prop: spannable}: using the compactness of $\Sig$, they showed that $\A$ is \textit{uniformly spannable}: there exists $k\in \N$ and $C_0>0$ such that $n_1,\ldots,n_d$ can be chosen to be at most $k$ and that a linear map $L \colon \R^d \to \R^d$ sending $\{v_i\}_{i=1}^d$ to an orthonormal basis of $\R^d$ satisfies $\|L\| < C_0$. 

\begin{prop}\label{prop: spannable implies qm}
Suppose a fiber-bunched cocycle $\A \in C^\alpha_b(\Sig,\glr)$ is uniformly spannable. Then $\A$ is quasi-multiplicative.
\end{prop}
\begin{proof}
Let $k$ and $C_0$ be the constants from uniform spannability of $\A$.
Given any $\I,\J \in \L$, let $\bar{x}\in [\I]$ be the point such that $\|\A(\I)\|= \|\A^{|\I|}(\bar{x})\|$, and set $x:=\s^{|\I|}(\bar{x})$. We similarly let $y \in [\J]$ such that $\|\A(\J)\|=\|\A^{|\J|}(y)\|$. Applying uniform spannability to $x,y \in \Sig$ and $u = u(\A^n(\bar{x}))$ gives vectors $\{v_i\}_{i=1}^d$ defined by \eqref{eq: v_i} that span $\R^d$. 

From the condition $\|L\|<C_0$ on a linear map $L$ straightening out $\{v_i\}_{i=1} ^d$ into an orthonormal basis, the $\rho$-distance (i.e., the angle) between each pair $v_i$ and $v_j$, $i \neq j$, is uniformly bounded below by some constant $\ep>0$ depending only on $C_0$. In particular, at least one of them, say $v_t$, satisfies $\rho(v(\A^{|\J|}(y))^\perp,v_t)>\ep/2$; this is similar to the angle gap obtained at the end of the proof for Proposition \ref{prop: typical implies QM}. 

Then Remark \ref{rem: linear algebra} applied to $A =  \A^{|\J|}(y)$, $C = H^s_{y_t,y} \A^{n_t}(x_t)  H^u_{x,x_t}$, and $B=\A^{|\I|}(\overline{x})$ gives 
\begin{align*}
\begin{split}
\|\A^{|\J|}(y)H^s_{y_t,y} \A^{n_t}(x_t)  H^u_{x,x_t}\A^{|\I|}(\overline{x})\| &\geq c\|\A^{|\J|}(y)\|\|\A^{|\I|}(\overline{x})\|\\
& = c\|\A(\I)\| \|\A(\J)\|
\end{split}
\end{align*}
for some $c>0$ that only depends on $\A$, $\ep$, and $k$. Denoting the
cylinder of length $n_t$ containing $x_t$ by $[\K]$, we have $\overline{x}_t:=\s^{-|\I|}x_t \in [\I\K\J]$ with $|\K| \leq k$.  Since the left hand side of the above inequality is equal to $\|H^s_{\s^{n_t+|\J|}x_t,\s^{|\J|}y}\A^{|\I|+n_t+|\J|}(\overline{x}_t)H^u_{\overline{x},\overline{x}_t}\|$, it is uniformly comparable to $\|\A(\I\K\J)\|$ due to the bounded distortion \eqref{eq: bdd distortion} of $\Phi_\A$ and \hol continuity of the canonical holonomies. This establishes the quasi-multiplicativity of $\A$.
\end{proof}

Since strong fiber-bunching is merely the usual fiber-bunching for $\gltwo$-cocycles, in view of Proposition \ref{prop: qm then unique eq} it is clear that Proposition \ref{prop: spannable} and \ref{prop: spannable implies qm} provide an alternative proof of Theorem \ref{thm: A}.	

\section{Application to cocycles over hyperbolic systems other than $(\Sig,\s)$}\label{sec: 6}
We describe how the results in this paper can be applied to certain cocycles over hyperbolic bases other than the subshift $(\Sig,\s)$, including Anosov diffeomorphisms of closed manifolds. Let $f$ be a $C^{1+\alpha}$ Anosov diffeomorphism of a closed manifold $M$. This means that there exist a $Df$-invariant splitting $TM = E^s \oplus E^u$ and constants $C>0$, $\nu \in (0,1)$ such that for all $n \in \N$,
$$\|Df^n|_{E^s}\| \leq C \nu^n ~\text{ and } ~\|Df^{-n}|_{E^u}\| \leq C\nu^n.$$
For a $C^{1+\alpha}$ Anosov diffeomorphism $f$, the stable and unstable bundles $E^s$ and $E^u$ are $\b$-\hol for some $\beta \in (0,\alpha]$.

Denoting the dimension of the unstable bundle $E^u$ by $d$, we then realize $Df|_{E^u}$ as a $\glr$-cocycle over a suitable subshift of finite type $(\Sig,\s)$ as follows: a Markov partition of $f$ \cite{bowen1975ergodic} results in a \hol continuous surjection $\pi \colon \Sig \to M$ such that $f \circ \pi= \pi \circ \s$. By choosing a Markov partition of sufficiently small diameter, we may assume that the image of each cylinder $[j]$ of $\Sig$, $1 \leq j \leq q$, is contained in an open set on which $E^u$ is trivializable. For $x\in [j]$, we let $L_j(x) \colon \R^d \to E^u_{\pi x}$ be a fixed trivialization of $E^u$ over $\pi([j])$. 
Supposing that $\s x \in [k]$, we define
\begin{equation}\label{eq: A lifted from Df}
\A(x):=L_k(\s x)^{-1} \circ D_{\pi x}f|_{E^u} \circ L_j(x).
\end{equation} 
This defines an $\alpha$-\hol $\glr$-cocycle $\A$ over a subshift $(\Sig,\s)$. 

We say the derivative cocycle $Df|_{E^u}$ is \textit{fiber-bunched} if there exists $N\in \N$ such that
$$\|D_xf^N|_{E^u_x}\| \cdot \|(D_xf^N|_{E^u_x})^{-1}\| \cdot \max\{\|D_xf^N|_{E^s_x}\|^\b,\|(D_xf^N|_{E^u_x})^{-1} \|^\b\}<1.$$
When $Df|_{E^u}$ is fiber-bunched, the canonical holonomies $H^{s/u}$ for $Df|_{E^u}$ converge and are $\b$-\hol continuous \cite{viana2008almost}. Similar to Definition \ref{defn: reducible} and \ref{defn: irreducible 2}, we say $Df|_{E^u}$ is \textit{reducible} if there exists a proper $Df|_{E^u}$-invariant and $H^{s/u}$-invariant sub-bundle of $E^u$, and \textit{irreducible} otherwise. 
Via the same Markov partition, the $\beta$-\hol canonical holonomies of $H^{s/u}$ of $Df|_{E^u}$ also lift to the $\beta$-\hol canonical holonomies (denoted again by) $H^{s/u}$ of $\A$ over $(\Sig,\s)$.

The following corollary translates Theorem \ref{thm: A} and \ref{thm: C} for the subadditive potential $\Phi_f$ over $M$ defined by
$$\Phi_{f}:=\{\log \|Df^n|_{E^u}\|\}_{n \in \N}.$$

\begin{cor}
Let $f$ be a transitive $C^{1+\alpha}$ Anosov diffeomorphism of a closed manifold $M$. Suppose that $\dim(E^u)=2$ and that $Df|_{E^u}$ is fiber-bunched. Then
\begin{enumerate}
\item If $Df|_{E^u}$ is irreducible, then $\Phi_f$ has a unique equilibrium state.
\item If $Df|_{E^u}$ is reducible, let $L$ be the $Df|_{E^u}$-invariant and $H^{s/u}$-invariant line bundle. Setting 
\begin{equation}\label{eq: a,c}
a(x):=Df|_{L_x} \text{ and } c(x) := \text{Jac}(Df|_{E^u_x})/a(x),
\end{equation}
$\Phi_f$ has a unique equilibrium state unless $\log|a|$ and $\log|c|$ satisfy two conditions from Theorem \ref{thm: C}, in which case there are exactly two ergodic equilibrium states. 
\end{enumerate}
\end{cor}
\begin{proof}
For the irreducible case, the proof follows the proof of Theorem \ref{thm: B} closely. If there exists a periodic point $p \in M$ of period $n\in \N$ such that $D_pf^{n}|_{E^u}$ has simple eigenvalues of distinct norms, then
either $Df^{n}|_{E^u}$ considered as a cocycle over $f^n$ is irreducible or $Df|_{E^u}$ interchanges two bi-holonomy invariant line bundles from Lemma \ref{lem: fixed}.

In the former case,
 Corollary \ref{cor: irred 1} applied to $f^n$ gives homoclinic points $z_{\pm} \in M$ of $p$ whose holonomy loops twist the eigendirections of $D_pf^{n}|_{E^u}$. The points in the subshift $\Sig$ corresponding to $p$ and $z_{\pm}$ ensure that the cocycle $\A^n$ over $(\Sig,\s^n)$ is weakly typical where $\A$ is defined as \eqref{eq: A lifted from Df}. Then Proposition \ref{prop: typical implies QM} and \ref{prop: qm then unique eq} give unique equilibrium state $\mu_\A \in \M(\s)$ of $\Phi_\A$. From the Gibbs property, $\mu_\A$ gives zero mass to $\pi^{-1}(\partial \mathcal{R})$ where $\partial \mathcal{R}$ is the union of the boundaries of the Markov partition (see \cite{bowen1975ergodic}), and hence, descends to the unique equilibrium state for $\Phi_f$. 

In the later case, two bi-holonomy invariant line bundles over $M$ lift to bi-holonomy line bundles over $\Sig$. Proceeding as in the proof of Theorem \ref{thm: A} in Section \ref{sec: thm A} gives a unique equilibrium state $\mu_\A \in \M(\s)$ for $\Phi_\A$ which is either also the unique equilibrium state for $\Phi_{\A^2}$ or an average of two distinct ergodic equilibrium states for $\Phi_{\A^2}$. The same reasoning as the previous paragraph (possibly applied with respect to $\s^2$ and the partition $\mathcal{R}\cap \s \mathcal{R}$) ensures that the projection of $\mu_\A$ is the unique equilibrium state for $\Phi_f$.

If there exists no such periodic point $p\in M$, then Proposition \ref{prop: conformal subbundle} applies to show that $Df|_{E^u}$ is conformal with respect to some $\alpha$-\hol continuous norm. By treating $\Phi_f$ as an additive potential over $M$ using conformality, Proposition \ref{prop: bowen} gives unique equilibrium state for $\Phi_f$. Note that this does not require passing to the subshift $(\Sig,\s)$.

For the reducible case, we pass to the subshift $(\Sig,\s)$ and the cocycle $\A$ defined as in \eqref{eq: A lifted from Df}. The $Df|_{E^u}$-invariant and $H^{s/u}$-invariant line bundle $L$ over $M$ lifts to the $\A$-invariant and $H^{s/u}$-invariant line bundle (also denoted by) $L$ over $\Sig$. By straightening out $L$, we obtain a $\beta$-\hol conjugacy $\CC \colon \Sig \to \gltwo$ of $\A$ into another cocycle $\B$ taking values in upper triangular matrices of the form \eqref{eq: B} with $a,c$ defined as in \eqref{eq: a,c}.
Then Theorem \ref{thm: C} provides criteria for there to be two distinct ergodic equilibrium states for $\Phi_\A$. The same reasoning as above shows that the equilibrium states for $\Phi_\A$ descend to equilibrium states for $\Phi_f$.
\end{proof}

\bibliographystyle{amsalpha}
\bibliography{gl2r_bib}
\end{document}